\documentclass[a4paper,12pt]{article}
\usepackage{amsmath}
\usepackage{amssymb}
\usepackage{amsthm}
\usepackage{latexsym}	
\usepackage{graphicx}
\usepackage{txfonts}
\usepackage{bm}
\usepackage{color}	

\newcommand{\qedJIAM}{\qed}
 \renewcommand{\qedJIAM}{}
\newcommand{\finboxARX}{\finbox}
\usepackage{geometry}
\numberwithin{equation}{section}

\newtheorem{theorem}{Theorem}[section]
\newtheorem{proposition}{Proposition}[section]
\newtheorem{lemma}{Lemma}[section]
\newtheorem{remark}{Remark}[section]
\newtheorem{example}{Example}[section]

\newcommand{\OMIT}[1]{{\bf [OMIT:} #1 \ {\bf --- end OMIT] }}  %%% For work
   \renewcommand{\OMIT}[1]{}            %%% For FINAL
\newcommand{\RR}{{\mathbb{R}}}
\newcommand{\ZZ}{{\mathbb{Z}}}

\newcommand{\vecone}{{\bf 1}}

\newcommand{\dom}{{\rm dom\,}}

\newcommand{\suppp}{{\rm supp}\sp{+}}
\newcommand{\suppm}{{\rm supp}\sp{-}}

\newcommand{\unitvec}[1]{e\sp{#1}}

\newcommand{\subgR}{\partial}
\newcommand{\subgZ}{\partial}
\newcommand{\argmin}{\arg \min}

\newcommand{\finbox}{\hspace*{\fill}$\rule{0.2cm}{0.2cm}$}
\newcommand{\todaye}{\the\year/\the\month/\the\day}

\newcommand{\Lnat}{{L$^{\natural}$}}
\newcommand{\Mnat}{{M$^{\natural}$}}
\newcommand{\LLnat}{{L$^{\natural}_{2}$}}
\newcommand{\MMnat}{{M$^{\natural}_{2}$}}
\newcommand{\LL}{{L$_{2}$}}
\newcommand{\MM}{{M$_{2}$}}

\def\calF{{\cal F}}

\newcommand{\Mvexb}{\mbox{\rm\bf (M-EXC)}}

\begin{document}

\title{Note on the Polyhedral Description of the Minkowski Sum of Two L-convex Sets%
% \thanks{
% This work was supported by JSPS KAKENHI Grant Numbers 
% JP17K00037, JP20K11697, JP21K04533.
% %% (Moriguchi=JP17K00037,JP21K04533; Murota=JP20K11697) 
% }%% thanks
%% \\ \RED{---After Reorganize--- (2022-02-17 evening)}
}%title

\author{
Satoko Moriguchi%
\thanks{
Faculty of Economics and Business Administration,
Tokyo Metropolitan University, 
%% Tokyo 192-0397, Japan, 
satoko5@tmu.ac.jp}
\ and 
Kazuo Murota%
\thanks{
The Institute of Statistical Mathematics,
%%Tokyo 190-8562, Japan; 
%% murota@ism.ac.jp; 
and Faculty of Economics and Business Administration,
Tokyo Metropolitan University, 
%%Tokyo 192-0397, Japan,
murota@tmu.ac.jp}
}%%author

\date{October 2021/November 2021/March 2022}
%%\date{October 2021/November 2021}

\maketitle

\begin{abstract}
{\rm L}-convex sets are one of the most fundamental concepts 
in discrete convex analysis.
Furthermore, the Minkowski sum of two L-convex sets,
called {\rm L}$_{2}$-convex sets,
is an intriguing object 
that is closely related to polymatroid intersection.
This paper reveals the polyhedral description of 
an {\rm L}$_{2}$-convex set,
together with the observation that 
the convex hull of an {\rm L}$_{2}$-convex set is a box-TDI polyhedron.
Two different proofs are given for the polyhedral description.
The first is a structural short proof,
relying on the conjugacy theorem in discrete convex analysis,
and the second is a direct algebraic proof,
based on Fourier--Motzkin elimination.
The obtained results admit natural graph representations.
Implications of the obtained results in discrete convex analysis
are also discussed.
\end{abstract}

{\bf Keywords}:
Discrete convex analysis, Fourier--Motzkin elimination,
{\rm L}-convex set, {\rm L}$_{2}$-convex set, Minkowski sum.
	%%Integer programming
%%\subclass{52A41 \and 90C10}
%%52A41=Convex functions and convex programs
%%90C10=Integer programming
%%90C27 = combinatorial optimization
%%90C25 = convex programming

%\newpage

%\tableofcontents
%%\newpage

%%\input{L2ineqMainbody}

%% file =  L2ineqMainbody.tex
%% 2021-09-10 / 2021-10-12 / 2021-10-18 / 2022-02-09 / 2022-03-04 / 2022-03-17

\section{Introduction}
\label{SCintro}

In discrete convex analysis (DCA),
{\rm L}-convex functions form one of the  most fundamental classes of discrete convex functions
\cite{Mdca98,Mdcasiam,Mbonn09}.
The concept of {\rm L}-convex functions, as well as 
their variant called \Lnat-convex functions,%
\footnote{%%%%%%%%%
 ``L''  stands for ``Lattice'' and ``\Lnat'' should be read ``ell natural.'' 
}  %%%%%% end of footnote %%%%.
sheds new light on algorithms in combinatorial optimization.
For example, Dijkstra's algorithm for shortest paths 
can be viewed as an instance of
L-convex function minimization \cite{MS14dijk}.
{\rm L}-convex functions have applications in several other fields including
computer vision \cite{Shi17L},
operations research (inventory theory, scheduling, etc.)
\cite{Che17,CL21dca,SCB14},
and economics and auction theory \cite{Mdcaeco16,Shi17L}.
Furthermore, the infimal convolution of two L-convex functions,
called {\rm L}$_{2}$-convex functions,
are the most intriguing objects 
in the duality theory in discrete convex analysis \cite[Chapter 8]{Mdcasiam}.
{\rm L}$_{2}$-convex functions are known to coincide with the conjugate
of {\rm M}$_{2}$-convex functions, the latter corresponding
to polymatroid intersection investigated in depth in combinatorial optimization
\cite{Sch03}. 

Concepts of discrete convex sets
are even more fundamental than those of discrete convex functions,
but at the same time, 
capture the essential properties of the corresponding discrete convex functions.
For example, the set of minimizers of an {\rm L}-convex function
is an {\rm L}-convex set.
Moreover, a function is {\rm L}-convex 
if and only if the set of minimizers of 
the function modified by an arbitrary linear function 
is always {\rm L}-convex
\cite[Section 7.5]{Mdcasiam}.
The set version of {\rm L}$_{2}$-convexity
is defined as the Minkowski sum (vector addition) of two L-convex sets.

The objective of this paper is to investigate 
the polyhedral description of \LL-convex sets.
That is, we aim at obtaining a system of inequalities
whose solution set coincides with the convex hull of a given \LL-convex set.
Such polyhedral description forms the basis 
of a standard approach in combinatorial optimization,
called polyhedral combinatorics \cite{Sch86,Sch03}.
Polyhedral descriptions are known for other kinds of discrete convex sets,
including {\rm L}-convex, \Lnat-convex,
{\rm M}-convex, and \Mnat-convex sets.
Polyhedral descriptions are also known for
\MM-convex and \MMnat-convex sets,
which correspond to polymatroid intersection.
In addition, the polyhedral description of multimodular sets 
has recently be obtained in \cite{MM21inclinter}.
It is worth noting that integrally convex sets 
\cite[Section 3.4]{Mdcasiam}
do not seem to admit a polyhedral characterization because
every set consisting of $\{ 0, 1\}$-vectors
is an integrally convex set, which fact implies that 
every $\{ 0, 1\}$-polytope is the convex hull of 
an integrally convex set.

In this paper we obtain polyhedral descriptions of an \LL-convex set,
together with the observation that 
the convex hull of an \LL-convex set is a box-TDI polyhedron.
Two versions of the polyhedral description of an \LL-convex set
are given,
Theorem~\ref{THl2polydesc} 
and 
Theorem~\ref{THl2polydescCyc}.
The former is a basic form, while the 
latter is a refinement 
with reference 
to a graph representation depending on the constituent {\rm L}-convex sets.
Although the basic form 
follows from its refinement in Theorem~\ref{THl2polydescCyc},
we give an independent short proof for 
Theorem~\ref{THl2polydesc}
relying on structural results in discrete convex analysis
such as the conjugacy theorem and \MM-optimality criterion.
Two different proofs are given to the refined form 
in Theorem~\ref{THl2polydescCyc}.
The first is a structural proof,
which is similar in vein to the proof of
Theorem~\ref{THl2polydesc}
but uses more detailed versions of 
the conjugacy theorem and \MM-optimality criterion.
The second is a direct algebraic proof,
which is based on Fourier--Motzkin elimination
applied to the combined system of inequalities
for the constituent {\rm L}-convex sets
and does not use any results from discrete convex analysis.
The obtained result (already Theorem~\ref{THl2polydesc})
has several implications in discrete convex analysis,
including an alternative proof
of the fundamental fact that a set of integer vectors is a box (interval)
if and only if 
it is both \LLnat-convex and \MMnat-convex.
The logical structure of the paper may be summarized by the diagram in Fig~\ref{FGlogic}. 

%%%  FIGURE %%%%%%%%%%%%%%%%%%
\begin{figure}
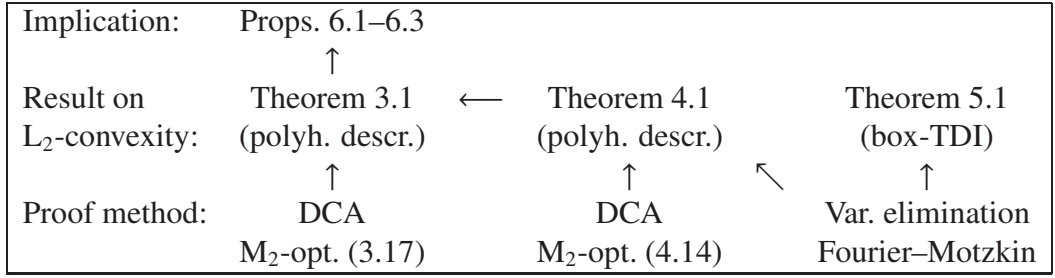

\centering 
\begin{tabular}{|lccccc|}
\hline
Implication: &  Props.~\ref{PRl2lnatB}--\ref{PRlnat2mnat2B} &  & & &
\\
 & $\uparrow$ & & & &
\\
Result on & Theorem~\ref{THl2polydesc}  & $\longleftarrow$  
 &  Theorem~\ref{THl2polydescCyc}  &    &   Theorem~\ref{THl2boxTDI}  
\\
\LL-convexity: & (polyh. descr.)  & 
 &  (polyh. descr.)  &    &   (box-TDI) 
\\
  & $\uparrow$ & & $\uparrow$ & $\nwarrow$ & $\uparrow$ 
\\ 
Proof method: & DCA & & DCA & & Var.~elimination
\\ 
 & \MM-opt.~\eqref{M2optcr1} & & \MM-opt.~\eqref{M2optCycle} & & Fourier--Motzkin
\\ \hline
\end{tabular}
\caption{Logical structure of this paper}
\label{FGlogic}
\end{figure}
%%%  FIGURE %%%%%%%%%%%%%%%%%%

This paper is organized as follows.
Section~\ref{SCprelim} recapitulates basic facts from discrete convex analysis,
focusing on L-convex and \LL-convex sets.
The main results are given in Sections \ref{SCpolydesc} and \ref{SCpolydescG}
with illustrative examples and structural proofs.
The proof by the Fourier--Motzkin elimination
is given in Section~\ref{SCelimappr}.
Applications of the obtained results 
in discrete convex analysis are shown in Section~\ref{SCimpli}. 
Finally, Section~\ref{SCconcl} concludes the paper with 
a summary of our present knowledge about the polyhedral description 
of discrete convex sets,
and Appendix gives definitions from discrete convex analysis.

%%\newpage

\section{Preliminaries}
\label{SCprelim}

Basic facts about L-convex and \LL-convex sets are 
introduced in this section.

\subsection{Basic notation}
\label{SCnotat}

Let $N = \{ 1,2,\ldots, n  \}$.
For a vector $x = (x_{1}, x_{2}, \ldots, x_{n})$
and a subset $I$ of $N$,
we use notation $x(I) = \sum_{i \in I} x_{i}$.
The inner product of two vectors $x$ and $y$ is denoted by
$\langle x,y \rangle$.
For a subset $I$ of $N$,
we denote by $\unitvec{I}$ the characteristic vector of $I$;
the $i$th component of $\unitvec{I}$ is equal to 1 or 0 according 
to whether $i \in I$ or not.
We use a short-hand notation $\unitvec{i}$ for $\unitvec{ \{ i \} }$,
which is the $i$th unit vector.

For vectors 
$a \in (\RR \cup \{ -\infty \})\sp{n}$ and 
$b \in (\RR \cup \{ +\infty \})\sp{n}$ 
with $a \leq b$,
the box (rectangle, interval)
between $a$ and $b$ is denoted by $[a,b]_{\RR}$,
i.e.,
\begin{equation*}  %%\label{boxdefR} 
[a,b]_{\RR} = \{ x \in \RR\sp{n} \mid a_{i} \leq x_{i} \leq b_{i} \ (i=1,2,\ldots,n)  \}.
\end{equation*}
For integer vectors 
$a \in (\ZZ \cup \{ -\infty \})\sp{n}$ and 
$b \in (\ZZ \cup \{ +\infty \})\sp{n}$ 
with $a \leq b$,
the box of integers 
between $a$ and $b$ 
means $[a,b]_{\RR} \cap \ZZ\sp{n}$.
The convex hull of a set $S$ $(\subseteq \ZZ\sp{n})$ 
is denoted by $\overline{S}$.

For two vectors $x, y \in \RR\sp{n}$,
the vectors of componentwise maximum and minimum of $x$ and $y$
are denoted, respectively, 
by $x \vee y$ and $x \wedge y$,
i.e.,
\begin{equation*}   %% \label{veewedgedef}
  (x \vee y)_{i} = \max(x_{i}, y_{i}),
\quad
  (x \wedge y)_{i} = \min(x_{i}, y_{i})
\qquad (i =1,2,\ldots, n).
\end{equation*}
The vector with all components equal to 1 is denoted by $\vecone$,
that is, 
$\vecone =(1,1,\ldots, 1) = \unitvec{N}$.

\subsection{L-convex sets}
\label{SCprelimL}

A nonempty set $S$ $(\subseteq \ZZ\sp{n})$ is called {\em L-convex} if 
it satisfies the following two conditions:
\begin{align} 
 x, y \in S
& \ \Longrightarrow \
 x \vee y, \ x \wedge y \in S ,
\label{submsetL}
\\ 
 x  \in S, \mu \in \ZZ
& \ \Longrightarrow \
 x + \mu {\bf 1} \in S .
\label{invarone}
\end{align}
The first condition \eqref{submsetL} means that $S$ forms a sublattice of $\ZZ\sp{n}$.
A polyhedron $P$ $(\subseteq \RR\sp{n})$ is called {\em L-convex} 
if 
\begin{align} 
 x, y \in P
& \ \Longrightarrow \
 x \vee y, \ x \wedge y \in P ,
\label{submsetLP}
\\ 
 x  \in P, \mu \in \RR
& \ \Longrightarrow \
 x + \mu {\bf 1} \in P .
\label{invaroneP}
\end{align}
The reader is referred to 
\cite[Proposition 2.5]{Msurvop19}  
for characterizations of L-convex sets.
The following polyhedral description of an L-convex set is known
\cite[Section 5.3]{Mdcasiam}.

\begin{theorem} \label{THlpolydesc}
\quad

\noindent
{\rm (1)}
A set $S \subseteq \ZZ\sp{n}$ is L-convex if and only if 
it can be represented as 
\begin{equation}  \label{SfromGamma}
S = \{ x \in \ZZ\sp{n} \mid x_{j} - x_{i} \leq \gamma_{ij} \ \ (i,j \in N) \}
\end{equation}
for some $\gamma_{ij} \in \ZZ \cup \{ +\infty \}$ $(i,j \in N )$ 
satisfying 
\begin{equation}  \label{gammaii}
\gamma_{ii} = 0
\qquad (i \in N)
\end{equation}
and 
the triangle inequality: 
\begin{equation}  \label{gammatriangle}
\gamma_{ij} + \gamma_{jk} \geq \gamma_{ik}
\qquad (i,j,k \in N).
\end{equation}
Such $\gamma_{ij}$ is determined from $S$ by%
\begin{equation}  \label{gammaFromS}
\gamma_{ij} =   \sup \{ x_{j} - x_{i} \mid x \in S \} 
\qquad (i,j \in N).
\end{equation}

\noindent
{\rm (2)}
A polyhedron $P \subseteq \RR\sp{n}$ is L-convex if and only if 
it can be represented as 
\begin{equation}  \label{PfromGamma}
P = \{ x \in \RR\sp{n} \mid x_{j} - x_{i} \leq \gamma_{ij} \ \ (i,j \in N) \}
\end{equation}
for some $\gamma_{ij} \in \RR \cup \{ +\infty \}$
$(i,j \in N )$
satisfying \eqref{gammaii} and \eqref{gammatriangle}.
Such $\gamma_{ij}$ is determined from $P$ by
\begin{equation}  \label{gammaFromP}
\gamma_{ij} =  \sup  \{ x_{j} - x_{i} \mid x \in P \} 
\qquad (i,j \in N).
\end{equation}
\finboxARX
\end{theorem}

\begin{remark} \rm \label{RMpolydescL}
Here are additional remarks about the polyhedral descriptions 
in Theorem~\ref{THlpolydesc}.

\begin{itemize}
\item
The correspondence between $S$ and integer-valued $\gamma$ with
\eqref{gammaii} and
\eqref{gammatriangle}
is bijective (one-to-one and onto)
through \eqref{SfromGamma} and \eqref{gammaFromS}.

\item
The correspondence between $P$ and real-valued $\gamma$ with 
\eqref{gammaii} and
\eqref{gammatriangle}
is bijective (one-to-one and onto)
through \eqref{PfromGamma} and \eqref{gammaFromP}.

\item
For any $\gamma_{ij} \in \RR \cup \{ +\infty \}$ $(i,j \in N )$
(independent of the triangle inequality),
$P$ in \eqref{PfromGamma}
is an L-convex polyhedron if $P \ne \emptyset$.
We have $P \ne \emptyset$ if and only if there exists no negative cycle 
with respect to $\gamma_{ij}$,
where a negative cycle means a set of indices $i_{1}, i_{2}, \ldots, i_{m}$ 
such that
$\gamma_{i_{1}i_{2}} +\gamma_{i_{2}i_{3}}+ \cdots
 +\gamma_{i_{m-1}i_{m}} +\gamma_{i_{m}i_{1}} < 0$.

\item
The convex hull of an L-convex set $S$ is an L-convex polyhedron
described by 
integer-valued $\gamma_{ij} \in \ZZ \cup \{ +\infty \}$ $(i,j \in N )$.

\item
For any integer-valued $\gamma_{ij} \in \ZZ \cup \{ +\infty \}$ $(i,j \in N )$,
$P$ in \eqref{PfromGamma} (if $P \ne \emptyset$)  is an integer polyhedron
and $S = P \cap \ZZ\sp{n}$ is an L-convex set with $\overline{S} = P$.
\finbox
\end{itemize}
\end{remark}

The intersection of an L-convex set with a coordinate hyperplane
is called an {\em \Lnat-convex set}.
That is,
a nonempty set $S \subseteq \ZZ\sp{n}$ is called \Lnat-convex if
\begin{equation}  \label{LnatfromL}
S =  \{ x \mid (x,0) \in T  \}
\end{equation}
for some L-convex set $T \subseteq  \ZZ\sp{n+1}$,
where $x \in \ZZ\sp{n}$ and $(x,0) \in \ZZ\sp{n+1}$ in \eqref{LnatfromL}.
For an \Lnat-convex set $S$ in $\ZZ\sp{n}$, the set 
\begin{equation}  \label{LfromLnat}
T =  \{ (x + \alpha \vecone,\alpha) \mid x \in S, \alpha \in \ZZ  \}
\end{equation}
is an L-convex in $\ZZ\sp{n+1}$. 
Thus the concepts of L-convex sets and \Lnat-convex sets 
are essentially equivalent.
Moreover, an L-convex set is \Lnat-convex.

Similarly, 
the intersection of an L-convex polyhedron with a coordinate hyperplane
is called an \Lnat-convex polyhedron.
That is,
a polyhedron $P \subseteq \RR\sp{n}$ is called \Lnat-convex if
\begin{equation}  \label{LnatfromLP}
P =  \{ x \mid (x,0) \in Q  \}
\end{equation}
for some L-convex polyhedron $Q \subseteq  \RR\sp{n+1}$.
For an \Lnat-convex polyhedron $P$ in $\RR\sp{n}$, the set 
\begin{equation}  \label{LfromLnatP}
Q =  \{ (x  + \alpha \vecone,\alpha) \mid x \in P, \alpha \in \RR  \}
\end{equation}
is an L-convex polyhedron in $\RR\sp{n+1}$. 
An L-convex polyhedron is \Lnat-convex.

The polyhedral description
of an \Lnat-convex set (resp., polyhedron)
can be obtained from Theorem~\ref{THlpolydesc}
with the aid of the relation 
\eqref{LnatfromL} (resp., \eqref{LnatfromLP}).

\begin{theorem} \label{THlnatpolydesc}
\quad

\noindent
{\rm (1)}
A set $S \subseteq \ZZ\sp{n}$ is \Lnat-convex if and only if 
it can be represented as 
\begin{equation}  \label{SnatfromGamma}
S = \{ x \in \ZZ\sp{n} \mid 
\alpha_{i} \leq  x_{i} \leq \beta_{i} \ \ (i\in N),  \ 
x_{j} - x_{i} \leq \gamma_{ij} \ \ (i,j \in N) \}
\end{equation}
for some $\alpha_{i} \in \ZZ \cup \{ -\infty \}$,
$\beta_{i} \in \ZZ \cup \{ +\infty \}$,
and $\gamma_{ij} \in \ZZ \cup \{ +\infty \}$
$(i,j \in N )$
with $\gamma_{ii} =0$ $(i \in N )$
such that 
$\tilde \gamma_{ij}$ defined
for $i,j \in N \cup \{ 0 \}$
by 
\begin{equation}  \label{tildeGamma}
\tilde \gamma_{00} =0, \qquad
\tilde \gamma_{ij} = \gamma_{ij}, \quad
\tilde \gamma_{i0} = -\alpha_{i}, \quad
\tilde \gamma_{0j} = \beta_{j}  \qquad
(i,j \in N)
\end{equation}
satisfies the triangle inequality:
\begin{equation}  \label{gammatriangleLnat}
\tilde \gamma_{ij} + \tilde \gamma_{jk} \geq \tilde \gamma_{ik}
\qquad (i,j,k \in N \cup \{ 0 \}).
\end{equation}
Such $\alpha_{i}$, $\beta_{i}$, $\gamma_{ij}$ are determined from $S$ by
\begin{align*} 
&\alpha_{i} =  \inf \{ x_{i} \mid x \in S \}, \quad
\beta_{i} = \sup \{ x_{i} \mid x \in S \}
\qquad (i \in N),
%% \label{alphabetaFromSnat}
\\ &
\gamma_{ij} =  \sup \{ x_{j} - x_{i} \mid x \in S \} 
\qquad (i,j \in N).
%% \label{gammaFromSnat}
\end{align*}

\noindent
{\rm (2)}
A polyhedron $P \subseteq \RR\sp{n}$ is \Lnat-convex if and only if 
it can be represented as 
\begin{equation*}  %%\label{PnatfromGamma}
P = \{ x \in \RR\sp{n} \mid 
\alpha_{i} \leq  x_{i} \leq \beta_{i} \ \ (i\in N),  \ 
x_{j} - x_{i} \leq \gamma_{ij} \ \ (i,j \in N) \}
\end{equation*}
for some $\alpha_{i} \in \RR \cup \{ -\infty \}$,
$\beta_{i} \in \RR \cup \{ +\infty \}$,
and $\gamma_{ij} \in \RR \cup \{ +\infty \}$
$(i,j \in N )$
with $\gamma_{ii} =0$ $(i \in N )$
such that 
$\tilde \gamma_{ij}$ defined
by \eqref{tildeGamma}
satisfies the triangle inequality \eqref{gammatriangleLnat}.
Such $\alpha_{i}$, $\beta_{i}$, $\gamma_{ij}$ are determined from $P$ by
\begin{align*} 
&\alpha_{i} =  \inf \{ x_{i} \mid x \in P \}, \quad
\beta_{i} =  \sup \{ x_{i} \mid x \in P \}
\qquad (i \in N),
%% \label{alphabetaFromPnat}
\\ &
\gamma_{ij} =   \sup  \{ x_{j} - x_{i} \mid x \in P \} 
\qquad (i,j \in N).
%% \label{gammaFromPnat}
\end{align*}
\finboxARX
\end{theorem}

The statements made in Remark~\ref{RMpolydescL}
can be adapted to \Lnat-convexity.
The reader is referred to 
\cite[Section 5.5]{Mdcasiam} 
and \cite[Proposition 2.3]{Msurvop19}  
for characterizations of \Lnat-convex sets.

\subsection{\LL-convex sets}
\label{SCprelimL2}

A nonempty set $S$ $\subseteq \ZZ\sp{n}$ is called 
{\em \LL-convex}  
(resp., {\em \LLnat-convex})  
if it can be represented as the Minkowski sum
(vector addition)
of two L-convex (resp., \Lnat-convex) sets
\cite[Section 5.5]{Mdcasiam}.
That is,
\begin{equation*} %% \label{L2setdef}
 S = \{ y + z \mid y \in S_{1} ,  z \in S_{2} \},
\end{equation*}
where
$S_{1}$ and $S_{2}$ are L-convex (resp., \Lnat-convex) sets.
Similarly,
a polyhedron $P$ $\subseteq \RR\sp{n}$
is called 
{\em \LL-convex}  
(resp., {\em \LLnat-convex})  
if it is the Minkowski sum of 
two L-convex (resp., \Lnat-convex) polyhedra.
An L-convex set is an \LL-convex set, but the converse is not true.
Similarly,
an \Lnat-convex set is \LLnat-convex, but the converse is not true.
An \LL-convex set (resp., polyhedron) is \LLnat-convex,
because L-convex sets (resp., polyhedra) are \Lnat-convex.

It is a basic fact
that an \LLnat-convex set (resp., polyhedron) 
is precisely the intersection of an L$_{2}$-convex set (resp., polyhedron) 
with a coordinate hyperplane,%
\footnote{%%%%%%%%%%%%%%%%%%%%
This fact is stated in \cite[p.129]{Mdcasiam} without a proof.
} %%% footnote %%%%%%%%%
which is proved in Proposition~\ref{PRfromL2toL2nat} below.
We first note a simple fact.

\begin{proposition} \label{PRminkowrestr}
\quad

\noindent
{\rm (1)}
For any $T \subseteq \ZZ\sp{n+1}$,
define $\varphi(T) := \{ x \mid (x,0) \in T  \}$.
If $T_{1}$ and $T_{2}$ have the property \eqref{invarone},
then  $\varphi(T_{1} + T_{2}) = \varphi(T_{1}) + \varphi(T_{2})$.

\noindent
{\rm (2)}
For any $Q \subseteq \RR\sp{n+1}$,
define $\varphi(Q) := \{ x \mid (x,0) \in Q  \}$.
If $Q_{1}$ and $Q_{2}$ have the property \eqref{invaroneP},
then  $\varphi(Q_{1} + Q_{2}) = \varphi(Q_{1}) + \varphi(Q_{2})$. 
\end{proposition}
\begin{proof}
We prove (1) only, while (2) can be proved in the same way. 
For any $T_{1}, T_{2} \subseteq  \ZZ\sp{n+1}$,
we have
\begin{align*}
\varphi(T_{1} + T_{2}) &= \{ x \mid (x,0) \in T_{1} + T_{2} \}
\nonumber \\
  &= \{ x \mid (x,0)= (y,\alpha) + (z,\beta), 
   \ (y,\alpha) \in T_{1}, \  (z,\beta) \in T_{2} \}
\nonumber \\
  &= \{ y+z \mid (y,\alpha) \in T_{1},  \ (z,-\alpha) \in T_{2} \},
%%\label{prfLLnat1}
\end{align*}
whereas
\begin{equation*}  %%\label{prfLLnat2}
\varphi(T_{1}) + \varphi(T_{2}) = \{ y+z  \mid (y,0) \in T_{1}, (z,0) \in T_{2} \}.
\end{equation*}
Therefore,
\begin{equation*}  %% \label{minkowrestG2}
 \varphi(T_{1}) + \varphi(T_{2})  \subseteq \varphi(T_{1} + T_{2}).
\end{equation*}
The reverse inclusion ($\supseteq $) holds under \eqref{invarone}.
Take any $x \in \varphi(T_{1} + T_{2})$.
Then there exist
$(y,\alpha) \in T_{1}$ and $(z,-\alpha) \in T_{2}$
satisfying $x = y+z$.
By \eqref{invarone},
we have
\[
 (y,\alpha) - \alpha (\vecone,1) = (y - \alpha \vecone, 0) \in T_{1},
\quad
 (z,-\alpha) + \alpha (\vecone,1) = (z + \alpha \vecone, 0) \in T_{2},
\]
from which follows that 
$x = y+z = (y - \alpha \vecone) + (z + \alpha \vecone) 
\in \varphi(T_{1}) + \varphi(T_{2})$.
\qedJIAM
\end{proof}

\begin{proposition} \label{PRfromL2toL2nat}
\quad

\noindent
{\rm (1)}
For each \LL-convex set $T \subseteq  \ZZ\sp{n+1}$,
$S :=  \{ x \in \ZZ\sp{n} \mid (x,0) \in T  \}$
is an \LLnat-convex set, and every \LLnat-convex set $S \subseteq \ZZ\sp{n}$ 
arises in this way.

\noindent
{\rm (2)}
For each \LL-convex polyhedron $Q \subseteq  \RR\sp{n+1}$,
$P :=  \{ x \in \RR\sp{n} \mid (x,0) \in Q  \}$
is an \LLnat-convex polyhedron, and every \LLnat-convex polyhedron $P \subseteq \RR\sp{n}$ 
arises in this way.
\end{proposition}

\begin{proof}
We prove (1) only, while (2) can be proved in the same way. 
Let $T$ be an \LL-convex set, which is represented as 
$T=T_{1}+T_{2}$ with two L-convex sets $T_{1}$ and $T_{2}$.
By the property \eqref{invarone} of L-convexity,
we have
$\varphi(T_{1} + T_{2}) = \varphi(T_{1}) + \varphi(T_{2})$
in the notation of Proposition~\ref{PRminkowrestr}.
This shows that $S=\varphi(T)=\varphi(T_{1} + T_{2})= \varphi(T_{1}) + \varphi(T_{2})$ 
is \LLnat-convex.
Conversely, let $S$ be an \LLnat-convex set.
By definition, $S$ can be represented as
$S=S_{1}+S_{2}$ with two \Lnat-convex sets $S_{1}$ and $S_{2}$.
Then $S_{1}=\varphi(T_{1})$  and $S_{2}=\varphi(T_{2})$  
for some L-convex sets $T_{1}$ and $T_{2}$.
Let $T= T_{1} + T_{2}$, which is \LL-convex.
It then follows from Proposition~\ref{PRminkowrestr} that 
$S= \varphi(T_{1}) + \varphi(T_{2}) = \varphi(T_{1} + T_{2}) = \varphi(T)$.
\qedJIAM
\end{proof}

\section{Polyhedral description of \LL-convex sets}
\label{SCpolydesc}

\subsection{Preliminary considerations}
\label{SCpolydescPre}

Let 
$S = S_{1} + S_{2} = \{ x \mid x = y + z, \  y \in S_{1}, z \in S_{2} \}$
be an \LL-convex set,
where $S_{1}$ and $S_{2}$ are L-convex sets.
By Theorem~\ref{THlpolydesc}, we can represent
$S_{1}$ and $S_{2}$ as
\begin{align} 
S_{1} &= \{ y \in \ZZ\sp{n} \mid 
  y_{j} - y_{i} \leq \gamma_{ij}\sp{(1)} \ \ ((i,j) \in E_{1}) \},
\label{S1descPre} 
\\
S_{2} &= \{ z \in \ZZ\sp{n} \mid 
  z_{j} - z_{i} \leq \gamma_{ij}\sp{(2)} \ \ ((i,j) \in E_{2}) \},
\label{S2descPre} 
\end{align}
where  $E_{1}, E_{2} \subseteq (N \times N) \setminus \{ (i,i) \mid i \in N \}$, 
$\gamma_{ij}\sp{(1)} \in \ZZ$ (finite-valued)
for all $(i,j) \in E_{1}$, and 
$\gamma_{ij}\sp{(2)} \in \ZZ$ for all $(i,j) \in E_{2}$.
We do not impose triangle inequality on 
$\gamma\sp{(k)} = ( \gamma_{ij}\sp{(k)} \mid (i,j) \in E_{k})$,
which is allowed  by Remark~\ref{RMpolydescL}.
The objective of this preliminary section is to derive 
a plausible system of inequalities to describe $S$.

In \eqref{S1descPre} and \eqref{S2descPre} we have
\begin{align} 
&
y_{j} - y_{i} \leq \gamma_{ij}\sp{(1)} 
 \quad ((i,j) \in E_{1}),
\label{S1descPre2} 
\\ &
 z_{j} - z_{i} \leq \gamma_{ij}\sp{(2)} 
\quad ((i,j) \in E_{2}).
\label{S2descPre2} 
\end{align}
Suppose that there are indices
$(i_{1}, j_{1}, \ldots, i_{m}, j_{m})$
such that
$(i_{r},j_{r}) \in E_{1}$
and
$(i_{r+1},j_{r}) \in E_{2}$
for $r=1,2,\ldots, m$, where
$i_{m+1} = i_{1}$.
By adding inequalities 
\eqref{S1descPre2} 
for $(i,j)=(i_{r},j_{r})$
with $r=1,2,\ldots, m$
and 
\eqref{S2descPre2} 
for $(i,j)=(i_{r+1},j_{r})$
with $r=1,2,\ldots, m$,
we obtain
\[
 \sum_{r=1}\sp{m} (y_{j_{r}}+z_{j_{r}}) -  \sum_{r=1}\sp{m} (y_{i_{r}}+z_{i_{r}}) 
\leq  \sum_{r=1}\sp{m} \gamma_{i_{r}j_{r}}\sp{(1)}  
 +  \sum_{r=1}\sp{m} \gamma_{i_{r+1}j_{r}}\sp{(2)}  .
\]
Since $y+z = x$, it follows that
\begin{equation}  \label{L2SdescPre}
 x(\{ j_{1}, \ldots, j_{m} \}) - x(\{ i_{1}, \ldots, i_{m} \}) 
\leq 
 \sum_{r=1}\sp{m} \gamma_{i_{r}j_{r}}\sp{(1)}  
 +  \sum_{r=1}\sp{m} \gamma_{i_{r+1}j_{r}}\sp{(2)}  .
\end{equation}
For any
$(i_{1}, j_{1}, \ldots, i_{m}, j_{m})$,
\eqref{L2SdescPre}
is a valid inequality for $S$, which means that
every element $x$ of $S$ satisfies \eqref{L2SdescPre}.
By choosing any family $\mathcal{F}$ of such indices 
$(i_{1}, j_{1}, \ldots, i_{m}, j_{m})$,
we obtain a system of inequalities indexed by $\mathcal{F}$, for which
we have
\begin{equation}  \label{L2Svalid}
 S \subseteq \{ x \in \ZZ\sp{n} \mid \mbox{\eqref{L2SdescPre}  
for all $(i_{1}, j_{1}, \ldots, i_{m}, j_{m}) \in \mathcal{F}$} \}.
\end{equation}

The inequality in \eqref{L2SdescPre}
has a characteristic feature that 
each coefficient of the variable $x$ belongs to $\{ -1,0,+1 \}$
and there are as many ``$+1$'' as ``$-1$'' among the coefficients.
The main message of this paper is that 
an \LL-convex set can indeed be described by such inequalities
with a suitable $\mathcal{F}$.

\subsection{Theorems}
\label{SCpolydescThmS}

The following theorem gives a polyhedral description of an \LL-convex set
(or polyhedron).
The constants $\gamma_{IJ}$ in
\eqref{L2Sdesc} and \eqref{L2Pdesc} 
will be determined 
in Theorem~\ref{THl2polydescCyc} in Section~\ref{SCpolydescG}.

\begin{theorem} \label{THl2polydesc}
\quad

\noindent
{\rm (1)}
An \LL-convex set $S \subseteq \ZZ\sp{n}$ can be represented as 
\begin{equation}  \label{L2Sdesc}
S = \{ x \in \ZZ\sp{n} \mid x(J) - x(I) \leq \gamma_{IJ} \ 
\mbox{\rm for all $(I,J)$ with $|I| =|J|$, $I \cap J = \emptyset$} \}
\end{equation}
for some $\gamma_{IJ} \in \ZZ \cup \{ +\infty \}$
indexed by $(I,J)$.

\noindent
{\rm (2)}
An \LL-convex polyhedron $P \subseteq \RR\sp{n}$ can be represented as 
\begin{equation}  \label{L2Pdesc}
P = \{ x \in \RR\sp{n} \mid x(J) - x(I) \leq \gamma_{IJ} \
\mbox{\rm for all $(I,J)$ with $|I| =|J|$, $I \cap J = \emptyset$} \}
\end{equation}
for some $\gamma_{IJ} \in \RR \cup \{ +\infty \}$
indexed by $(I,J)$.
\end{theorem}
\begin{proof}
The proof is given in Section~\ref{SCpolydescprf}.
\qedJIAM
\end{proof}

The following example shows that the number 
of variables in an inequality $x(J) - x(I) \leq \gamma_{IJ}$,
which is $|I|+|J|$,  is not bounded by a constant.

\begin{example} \rm \label{EXlongineq}
Let $n = 2m$ be an even integer, and consider 
the \LL-convex set $S = S_{1} + S_{2}$ defined by L-convex sets
\begin{align} 
 S_{1} &= \{ y \in \ZZ\sp{n}  \mid 
  y_{2i-1} \leq  y_{2i} \ \  (i=1,2,\ldots, m)  \},
\label{longineqS1}
\\
 S_{2} &= \{ z \in \ZZ\sp{n}  \mid 
  z_{2i+1} \leq  z_{2i} \ \  (i=1,2,\ldots, m) \},
\label{longineqS2}
\end{align}
where $z_{n+1} = z_{1}$.
The description of $S = S_{1} + S_{2}$ is given by
\begin{equation} \label{longineqS}
 S = \{ x \in \ZZ\sp{n}  \mid 
 (x_{1} + x_{3} + \cdots + x_{n-1}) - (x_{2} + x_{4} + \cdots + x_{n}) \leq 0   \},
\end{equation}
which is shown later in Example~\ref{EXlongineqG} in Section~\ref{SCthmrefined}.
All variables are involved in a single inequality.
\finbox
\end{example}

By Proposition \ref{PRfromL2toL2nat},
an \LLnat-convex set is nothing but the intersection of 
an L$_{2}$-convex set with a coordinate hyperplane.
Hence Theorem~\ref{THl2polydesc}
immediately implies the following theorem
for an \LLnat-convex set (or polyhedron).

\begin{theorem} \label{THlnat2polydesc}
\quad

\noindent
{\rm (1)}
An \LLnat-convex set $S \subseteq \ZZ\sp{n}$ can be represented as 
\begin{align} 
S =  \{ x \in \ZZ\sp{n} \mid & \
 x(J) - x(I) \leq \gamma_{IJ} \  
\mbox{\rm for all $(I,J)$} 
\nonumber \\ & \ 
\mbox{\rm with $|I|-|J| \in \{ -1,0,1 \}$, $I \cap J = \emptyset$} \}
 \label{Lnat2Sdesc}
\end{align}
for some $\gamma_{IJ} \in \ZZ \cup \{ +\infty \}$
indexed by $(I,J)$.

\noindent
{\rm (2)}
An \LLnat-convex polyhedron $P \subseteq \RR\sp{n}$ can be represented as 
\begin{align}  
P = \{ x \in \RR\sp{n} \mid & \
 x(J) - x(I) \leq \gamma_{IJ} \
\mbox{\rm for all $(I,J)$} 
\nonumber \\ & \ 
\mbox{\rm with $|I|-|J| \in \{ -1,0,1 \}$, $I \cap J = \emptyset$} \}
\label{Lnat2Pdesc}
\end{align}
for some $\gamma_{IJ} \in \RR \cup \{ +\infty \}$
indexed by $(I,J)$.
\finboxARX
\end{theorem}

\begin{example} \rm \label{EXLnat2poly}
Consider
$ S = \{ (0,0,0), (1,1,0), (0,1,1), (1,2,1) \}$,
which is \LLnat-convex (but not \Lnat-convex).
Indeed we have 
$S = S_{1} + S_{2}$
with two \Lnat-convex sets
$S_{1}= \{ (0,0,0), (1,1,0) \}$ and $S_{2} = \{ (0,0,0), (0,1,1) \}$.
(This is an example taken from \cite[Example 3.11]{MS01rel}.)
All four points of $S$ lie on the hyperplane
$x_{1} - x_{2} + x_{3} = 0$,
and it is easy to see, by inspection, that 
\begin{align*} 
 S & = \{ x \in \ZZ\sp{3} 
          \mid  x_{1} - x_{2} + x_{3} = 0, \ 0 \leq x_{1} \leq  1, \ 0 \leq x_{3} \leq  1 \}
\\ & 
= \{ x \in \ZZ\sp{3} 
      \mid  x_{1} - x_{2} + x_{3} \leq 0, \ -x_{1} + x_{2} - x_{3} \leq 0, 
\\ & \phantom{= \{ x \in \ZZ\sp{3} \mid} \, 
    \ x_{1} \leq  1, \ -x_{1} \leq  0,
    \ x_{3} \leq  1, \ -x_{3} \leq  0 \}
\end{align*}
gives a polyhedral description of the form of \eqref{Lnat2Sdesc}
in Theorem~\ref{THlnat2polydesc}.
\finbox
\end{example}

\begin{remark} \rm \label{RMedgeM2}
The inequality $x(J) - x(I) \leq \gamma_{IJ}$ 
in \eqref{L2Pdesc}
can be rewritten as 
$\langle \unitvec{J} - \unitvec{I}, x \rangle \leq \gamma_{IJ}$, 
where
$\langle \cdot,  \cdot \rangle$ denotes the inner product,
and $\unitvec{I}$ and $\unitvec{J}$ are the characteristic vectors of 
$I$ and $J$, respectively. 
This shows that if the polyhedron $P$ is full-dimensional,
the normal vector of a facet
of $P$ is of the form of 
$c (\unitvec{J} - \unitvec{I})$  with $c \ne 0$.
\finbox
\end{remark}

\begin{remark} \rm \label{RMnormalfan}
Theorem~\ref{THl2polydesc} as well as Theorem~\ref{THlnat2polydesc}
is consistent with the general result 
\cite[Proposition 7.12]{Zie07} 
on convex polytopes (bounded polyhedra)
that
the normal fan of the Minkowski sum of two polytopes
$P_{1}$ and $P_{2}$
is the common refinement of the individual fans, 
which means that each normal cone of $P=P_{1} + P_{2}$ 
is the intersection of a normal cone of $P_{1}$ and that of $P_{2}$.
By Theorem~\ref{THlpolydesc},
every normal cone of $P_{k}$ $(k=1,2)$ is spanned by vectors 
of the form $\unitvec{j} - \unitvec{i}$.
These two facts, when combined, indicate 
that each normal cone of $P$ is spanned by vectors 
of the form 
$\sum_{(i,j) \in K} c_{ij}(\unitvec{j} - \unitvec{i})$
with coefficients $c_{ij} \in \RR$
for some set $K$ of pairs $(i,j)$,
where
$I := \{ i \mid (i,j) \in K \}$ and $J := \{ j \mid (i,j) \in K \}$
may not be disjoint.
The expression \eqref{L2Pdesc} shows
that we can take $c_{ij}=1$,
from which follows that $I \cap J = \emptyset$ can be assumed.
\finbox
\end{remark}

\subsection{Proof of Theorem~\ref{THl2polydesc}}
\label{SCpolydescprf}

In this section we prove Theorem~\ref{THl2polydesc}
for an \LL-convex set (or polyhedron)
using concepts and results from discrete convex analysis \cite{Mdcasiam}.
The reader is referred to Appendix \ref{SClmfn}
for the definitions of \LL-convex and \MM-convex functions.

For a set $T \subseteq \ZZ\sp{n}$, in general, the {\em indicator function}
$\delta_{T}: \ZZ\sp{n} \to \{ 0, +\infty \}$
is defined by
\[
\delta_{T}(x)  =
   \left\{  \begin{array}{ll}
    0            &   (x \in T) ,      \\
   + \infty      &   (x \not\in T).  \\
                      \end{array}  \right.
\]
For an integer-valued function 
$h: \ZZ\sp{n} \to \ZZ \cup \{ +\infty \}$
(with $h(x_{0}) < +\infty$ for some $x_{0} \in \ZZ\sp{n}$),
the (integral) {\em conjugate function} 
$h\sp{\bullet}: \ZZ\sp{n} \to \ZZ \cup \{ +\infty \}$ 
is defined by
\begin{align}
h\sp{\bullet}(p)  &= \sup\{  \langle p, x \rangle - h(x)   \mid x \in \ZZ\sp{n} \}
\qquad ( p \in \ZZ\sp{n}).
 \label{conjvexZpZ} 
\end{align}
The  (integral)  {\em subdifferential} of $h$ at $x$ is defined by
\begin{equation} \label{subgZRdef0}
 \subgZ h(x)
= \{ p \in  \ZZ\sp{n} \mid    
  h(y) - h(x)  \geq  \langle p, y - x \rangle   \ \ \mbox{\rm for all }  y \in \ZZ\sp{n} \}.
\end{equation}
A vector $p$ belonging to $\subgR h(x)$
is called an  (integral) {\em subgradient} of $h$ at $x$.

Let $S \subseteq \ZZ\sp{n}$ be an \LL-convex set,
and denote its indicator function by $g$, that is, 
$g = \delta_{S}: \ZZ\sp{n} \to  \{ 0, +\infty \}$. 
Since $S$ is an \LL-convex set, $g$
is an \LL-convex function.
Let $f$ denote the conjugate function of 
$g$, that is, $f = g\sp{\bullet}$.
By the conjugacy relation
between \LL-convexity and \MM-convexity
(\cite[Theorem~8.48]{Mdcasiam}),
the function
$f: \ZZ\sp{n} \to \ZZ \cup \{ +\infty \}$
is an \MM-convex function, and $g = f\sp{\bullet}$.
In addition, $f(\bm{0})=0$ and
$f$ is positively homogeneous, since it is the conjugate of an indicator function.
We make use of a fundamental relation 
\begin{equation} 
 S =  \subgZ f(\bm{0})
 = \{ p \in  \ZZ\sp{n}  \mid 
    \bm{0} \in \argmin_{z} \{ f(z) -  \langle p, z \rangle \} \ \},
\label{M2SisSubgf}
\end{equation}
which can be proved as
\begin{equation*} %% \label{SisSubgf}
 \subgZ f(\bm{0}) 
 = \{ p  \mid    
  f(y)  \geq  \langle p, y \rangle   \ (\forall  y \in \ZZ\sp{n}) \}
 = \{ p \mid  f\sp{\bullet}(p) = 0 \}
 = \{ p \mid  g(p) = 0 \}
 = S.
\end{equation*}
See Section~8.1.3 (in particular, (8.17) and Fig.~8.1) of \cite{Mdcasiam}
for the correspondence between indicator functions
and positively homogeneous convex functions in discrete convex analysis.

In \eqref{M2SisSubgf}, the function
$f(z) -  \langle p, z \rangle$
is \MM-convex.
For a general \MM-convex function
$h: \ZZ^{n} \to \RR \cup \{ +\infty  \}$,
it is known as \MM-optimality criterion
\cite[Theorem~8.32]{Mdcasiam} that 
a vector $z^{*} \in \ZZ\sp{n}$
with $h(z^{*}) < +\infty$
is a minimizer of $h$
if and only if
\begin{equation} \label{M2optcr1}
h(z^{*}) \leq h(z^{*} +  \unitvec{J} - \unitvec{I}) 
\end{equation}
for all $(I,J)$ with $|I| =|J|$ and $I \cap J = \emptyset$.
This condition \eqref{M2optcr1}
for $h(z) = f(z) -   \langle p, z \rangle$ and $z^{*} =\bm{0}$
reads
\begin{equation*}  
  \langle \unitvec{J} - \unitvec{I}, p \rangle \leq 
 f(\unitvec{J} - \unitvec{I}).
\end{equation*}
Combining this with \eqref{M2SisSubgf}, we obtain
\begin{equation*} %%\label{L2SdescPrf}
 S= \{ p \in  \ZZ\sp{n} \mid    
  \langle \unitvec{J} - \unitvec{I}, p \rangle \leq 
 f(\unitvec{J} - \unitvec{I}),
  \ \mbox{$\forall (I,J)$: $|I|=|J|$, $I \cap J = \emptyset$}  \} .
\end{equation*}
This gives the desired polyhedral description in \eqref{L2Sdesc}
with $\gamma_{IJ} = f(\unitvec{J} - \unitvec{I})$.

The expression \eqref{L2Pdesc} for an \LL-convex polyhedron
can be established in a similar manner
by using the polyhedral version of the conjugacy relation
between \LL-convexity and \MM-convexity,
which can be derived easily from 
the conjugacy between polyhedral
{\rm L}-convex and {\rm M}-convex functions
in \cite[Theorem~8.4]{Mdcasiam}.
Theorem~\ref{THl2polydesc} also follows from Theorem~\ref{THl2polydescCyc} 
to be established later.

\begin{remark} \rm \label{RMdescL2adj}
The number of inequalities necessary to describe an \LL-convex set
in $\ZZ\sp{n}$ can be exponential in $n$.
This can be seen as follows.
Consider a matroid intersection problem,
and let $B$ be a common base.
The set of the characteristic vectors of common bases
is an \MM-convex set contained in $\{ 0,1 \}\sp{n}$.
Let $Q$ denote the convex hull of this \MM-convex set.
The tangent cone of $Q$ at $\unitvec{B}$ (the characteristic vector of $B$),
to be denoted by $Q_{B}$,
is an \MM-convex polyhedron
and the extreme rays of tangent cone $Q_{B}$ correspond to common bases adjacent to $B$.
The adjacency relation in matroid intersection has been investigated 
in \cite{FT88,Iwa02adj},
and an instance of a common base with exponentially many adjacent common bases
has been constructed in \cite{Iwa21adj}. 
Let $B$ be such a common base with exponentially many adjacent common bases.
Then the tangent cone $Q_{B}$ has exponentially many extreme rays.
Next consider the dual cone of $Q_{B}$, and call it $P$.
By the conjugacy between \MM-convexity and \LL-convexity
(\cite[Theorem~8.48]{Mdcasiam}),
$P$ is an \LL-convex polyhedron,
and its facets correspond to extreme rays of $Q_{B}$.
Moreover, $P$ is an integral polyhedron, implying
that $P$ is the convex hull of an \LL-convex set $S = P \cap \ZZ\sp{n}$.
It follows that the description of this \LL-convex set $S$ 
requires exponentially many inequalities.
\finbox
\end{remark}

\section{Refinement of the polyhedral description}
\label{SCpolydescG}

In Theorems \ref{THl2polydesc} and \ref{THlnat2polydesc} we have identified
inequalities of the form
$x(J) - x(I) \leq \gamma_{IJ}$ 
to describe \LL-convex and \LLnat-convex sets (and polyhedra).
In this section we establish their refinements
in Theorems \ref{THl2polydescCyc} and \ref{THlnat2polydescCyc} 
with the aid of a graph representation.

\subsection{Graph representations}
\label{SCgraphrep}

Let $S = S_{1}+S_{2}$ be an \LL-convex set
with two L-convex sets $S_{1}$ and $S_{2}$.
By Theorem~\ref{THlpolydesc},
each L-convex polyhedron $S_{k}$
$(k=1,2)$ 
is described as
\begin{equation} 
S_{k} = \{ y \in \ZZ\sp{n} \mid 
  y_{j} - y_{i} \leq \gamma_{ij}\sp{(k)} \ \ ((i,j) \in E_{k}) \},
\label{PkdescG} 
\end{equation}
where  $E_{k} \subseteq (N \times N) \setminus \{ (i,i) \mid i \in N \}$
and 
$\gamma_{ij}\sp{(k)} \in \ZZ$ (finite-valued)
for all $(i,j) \in E_{k}$.
We do not impose triangle inequality on 
$\gamma\sp{(k)} = ( \gamma_{ij}\sp{(k)} \mid (i,j) \in E_{k})$,
which is allowed  by Remark~\ref{RMpolydescL}.

With reference to \eqref{PkdescG} we consider a directed graph
$G_{k}=(N,E_{k})$ with vertex set 
$N = \{ 1,2,\ldots, n  \}$
and edge set $E_{k}$.
Each edge $(i,j) \in E_{k}$ is associated with a length of $\gamma_{ij}\sp{(k)}$.
We denote the reorientation of $G_{2}$
by $G_{2}\sp{\circ}=(N,E_{2}\sp{\circ})$,
where an edge $(i,j)$ of $G_{2}\sp{\circ}$
has length $\gamma\sp{(2)}_{ji}$.
The union of 
$G_{1}$ and $G_{2}\sp{\circ}$ is denoted by 
$G_{1} + G_{2}\sp{\circ} = (N, E_{1} \cup E_{2}\sp{\circ})$
or simply by
$G_{*} = (N, E_{*})$
with
$E_{*} =  E_{1} \cup E_{2}\sp{\circ}$.
Parallel edges may exist in $G_{*} = G_{1} + G_{2}\sp{\circ}$.
When necessary, an edge connecting $i$ to $j$
in $E_{1}$ (resp., $E_{2}\sp{\circ}$)
is denoted by 
$(i,j)_{1}$ 
(resp., $(i,j)_{2}$).
That is,
 $(i,j)_{1} \in E_{1}$ and $(i,j)_{2} \in E_{2}\sp{\circ}$.
The edge length $\gamma$ in $G_{*} = (N, E_{*})$ is defined for $e \in E_{*}$ by
\begin{equation} \label{edgelenG}
\gamma(e) = 
\begin{cases}
 \gamma\sp{(1)}_{ij} & \text{($e=(i,j)_{1} \in E_{1}$)}, \\
 \gamma\sp{(2)}_{ji} & \text{($e=(i,j)_{2} \in E_{2}\sp{\circ}$)}. 
\end{cases}
\end{equation}

For $k=1,2$, each graph $G_{k}$ contains no negative cycles 
by $S_{k} \ne \emptyset$
(Remark~\ref{RMpolydescL}).
Let $\lambda(i,j;G_{k})$ denote the minimum $\gamma\sp{(k)}$-length 
of a path connecting $i$ to $j$ in $G_{k}$, where 
$\lambda(i,j;G_{k}) = +\infty$
if there is no such path. 
Define $\lambda(j,i;G_{2}\sp{\circ}) := \lambda(i,j;G_{2})$.

\begin{example} \rm \label{EXl2dim4polyG}
Let $S = S_{1} + S_{2}$ be an \LL-convex set defined by 
\begin{align*} 
 S_{1} &= \{ y \in \ZZ\sp{4}  \mid 
 y_{2} - y_{1} \leq 3, \
 y_{3} - y_{2} \leq 5, \
 y_{4} - y_{3} \leq 8, \
 y_{1} - y_{4} \leq 7 
 \},
\\
 S_{2} &= \{ z \in \ZZ\sp{4}  \mid 
 z_{1} - z_{3} \leq 2, \
 z_{4} - z_{1} \leq 1, \
 z_{2} - z_{3} \leq 3, \
\nonumber \\ &  \phantom{= \{ z \in \ZZ\sp{4}  \mid} \ \,  
 z_{4} - z_{2} \leq 5, \
 z_{3} - z_{4} \leq 2
 \}.
%%\label{l2dim4S2}
\end{align*}
The graphs $G_{1}$ and $G_{2}$ 
associated with $S_{1}$ and $S_{2}$
are illustrated in Fig.~\ref{FGL2graph4}, where
vertex $i$ is shown by $\framebox{$i$}$ and
\begin{align*} 
& \gamma_{12}\sp{(1)} = 3, \quad  
\gamma_{23}\sp{(1)} = 5, \quad 
\gamma_{34}\sp{(1)} = 8, \quad
\gamma_{41}\sp{(1)} = 7;  
%% \label{l2dim4S1gamma}
\\ &
\gamma_{31}\sp{(2)} = 2, \quad  
\gamma_{14}\sp{(2)} = 1, \quad 
\gamma_{32}\sp{(2)} = 3, \quad
\gamma_{24}\sp{(2)} = 5, \quad
\gamma_{43}\sp{(2)} = 2.  
%% \label{l2dim4S2gamma}
\end{align*}
The graph $G_{*} = G_{1} + G_{2}\sp{\circ}$ is also shown.
\finbox
\end{example}

%%%  FIGURE %%%%%%%%%%%%%%%%%%
%%\input{fgL2gr4.tex}
\begin{figure}\begin{center}
\includegraphics[height=40mm]{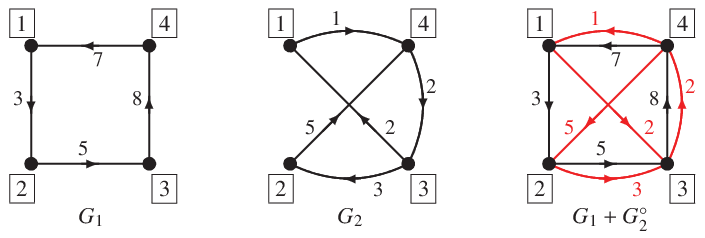}
\caption{Graphs for {\rm L}-convex polyhedra (Example~\ref{EXl2dim4polyG})} 
\label{FGL2graph4}
\end{center}\end{figure}
%%%  FIGURE %%%%%%%%%%%%%%%%%%

We represent a directed cycle in $G_{*}$ 
as a (cyclic) sequence of edges:
$C = e_{1} e_{2} \cdots e_{\ell-1} e_{\ell}$,
where $e_{k} \in E_{*}$ 
for $k=1,2,\ldots,\ell$ and 
the head (terminal vertex) of $e_{k}$ is the tail (initial vertex)
of $e_{k+1}$ for $k=1,2,\ldots,\ell$ 
with the convention of  
$e_{\ell +k}=e_{k}$.
When we speak of a cycle, 
we always mean a directed cycle in this paper. 
We denote the length of $C$ by
\begin{equation} 
\gamma(C) := \sum_{k=1}\sp{\ell} \gamma(e_{k})
 = \sum_{(i,j)_{1} \in C} \gamma_{ij}\sp{(1)} 
 + \sum_{(i,j)_{2} \in C} \gamma_{ji}\sp{(2)} .
\label{gammaCGdef}
\end{equation}
A cycle $C$ is called {\em simple}
if the vertices on $C$ are all distinct.

We call a cycle 
{\em mixed}
if it contains edges from 
both $E_{1}$ and $E_{2}\sp{\circ}$.
For a mixed cycle $C$,
we define its {\em break vertices} 
$( i_{1}, j_{1}, \ldots, i_{m}, j_{m})$
and an inequality associated with $C$ by
\begin{equation}
x(\{ j_{1}, \ldots, j_{m} \}) - x(\{ i_{1}, \ldots, i_{m} \}) 
  \leq  \gamma(C),
\label{L2ineqCyc}
\end{equation}
which is used to formulate Theorem~\ref{THl2polydescCyc}.

The break vertices 
$( i_{1}, j_{1}, \ldots, i_{m}, j_{m})$
are defined as follows.
Assume, without loss of generality, that
a mixed cycle 
$C$ is represented as
$C = e_{1} e_{2} \cdots e_{\ell-1} e_{\ell}$ with
$e_{1} \in E_{1}$ and $e_{\ell} \in E_{2}\sp{\circ}$,
which is possible by a cyclic permutation of the edges in $C$.
The first break vertex $i_{1}$ is the tail of $e_{1}$,
which is also the head of $e_{\ell}$.
In general, a break vertex 
$i_{r}$ is the tail of $e_{k} \in C \cap E_{1}$ 
and the head of $e_{k-1} \in C \cap E_{2}\sp{\circ}$
for some $k \ (=: s_{r})$,
whereas 
$j_{r}$ is the head of $e_{k'} \in C \cap E_{1}$ 
and the tail of $e_{k'+1} \in C \cap E_{2}\sp{\circ}$
for some $k' \ (=: t_{r})$.
Including the break vertices we may represent $C$ as
\begin{align}  
C = & \ 
 |i_{1}| \underbrace{e_{1} \cdots e_{t_{1}}}_{E_{1}} |j_{1}|
    \underbrace{e_{t_{1}+1} \cdots e_{s_{2}-1}}_{E_{2}\sp{\circ}}
 |i_{2}| \underbrace{e_{s_{2}} \cdots e_{t_{2}}}_{E_{1}} |j_{2}| 
  e_{t_{2}+1} \cdots
\nonumber \\ & \ 
\cdots e_{s_{m}-1}  
 |i_{m}| \underbrace{e_{s_{m}} \cdots e_{t_{m}}}_{E_{1}}
 |j_{m}| \underbrace{e_{t_{m}+1} \cdots e_{\ell}}_{E_{2}\sp{\circ}} |i_{1}| .
\label{dircycBk}
\end{align}
Note that $s_{1} = 1$ and
$s_{r} \leq  t_{r} \leq s_{r+1} -2$
for $r=1,2,\ldots,m$ with $s_{m+1} = \ell+ 1$,
and that 
$e_{k} \in E_{1}$ if $s_{r} \leq k \leq  t_{r}$
and 
$e_{k} \in E_{2}\sp{\circ}$ if $t_{r}+1 \leq k \leq  s_{r+1}-1$.
For $r=1,2,\ldots,m$,
an interval of $E_{1}$-edges 
starts at vertex $i_{r}$ and ends at vertex $j_{r}$,
and an interval of $E_{2}\sp{\circ}$-edges 
starts at vertex $j_{r}$ and ends at vertex $i_{r+1}$.
We often refer to an interval of $E_{1}$-edges
(resp., $E_{2}\sp{\circ}$-edges) 
as an $E_{1}$-interval
(resp., $E_{2}\sp{\circ}$-interval).
The index $m$ is equal to the number of $E_{1}$-intervals in $C$,
which is also equal to 
the number of $E_{2}\sp{\circ}$-intervals.

\begin{example} \rm \label{EXl2graphcyc}
The concepts introduced above are illustrated 
for $G_{*} = G_{1} + G_{2}\sp{\circ}$ in Fig.~\ref{FGL2graph4}. 

\begin{itemize}
\item
For a simple mixed cycle 
$C  = (1,2)_{1} (2,3)_{2} (3,4)_{2} (4,1)_{2}$,
there is a pair of break vertices $(i_{1}, j_{1})=(1,2)$.
We have 
$\gamma(C) = 3 + (3+2+1)=9$
and $x_{2} - x_{1} \leq 9$ 
in \eqref{L2ineqCyc}.

\item
For a simple mixed cycle 
$C  = (1,2)_{1} (2,3)_{1} (3,4)_{2} (4,1)_{2}$,
there is a pair of break vertices $(i_{1}, j_{1})=(1,3)$.
We have $\gamma(C) = (3+5) + (2+1)=11$
and
$x_{3} - x_{1} \leq 11$
in \eqref{L2ineqCyc}.

\item
For a simple mixed cycle 
$C = (1,2)_{1} (2,3)_{2} (3,4)_{1} (4,1)_{2}$,
the break vertices are given by
$(i_{1}, j_{1})=(1,2)$ and $(i_{2}, j_{2})=(3,4)$.
We have
$\gamma(C) =  3 + 3 + 8 + 1 = 15$
and
$x(\{ 2,4 \}) - x(\{ 1,3 \}) \leq 15$
in \eqref{L2ineqCyc}.

\item
For a simple mixed cycle 
$C  = (3,4)_{1} (4,1)_{2} (1,3)_{2}$,
there is a pair of break vertices $(i_{1}, j_{1})=(3,4)$.
We have $\gamma(C) = 8 + (1+2)=11$
and $x_{4} - x_{3} \leq 11$
in \eqref{L2ineqCyc}.

\item
A simple mixed cycle 
$C = (3,4)_{1} (4,2)_{2} (2,3)_{2}$
has the same break vertices 
$(i_{1}, j_{1})=(3,4)$ as above
and a longer length with $\gamma(C) = 8 + (5+3)=16$.
This results in $x_{4} - x_{3} \leq 16$ in \eqref{L2ineqCyc},
which is weaker than (implied by) 
$x_{4} - x_{3} \leq 11$ in the above.

\item
A non-simple mixed cycle 
$C = (2,3)_{1} (3,4)_{1} (4,1)_{1} (1,3)_{2} (3,4)_{2} (4,2)_{2}$
has a pair of break vertices
$(i_{1}, j_{1})=(2,1)$,
and 
$\gamma(C) = (5+8+7) + (2+2+5) = 29$.
The associated inequality
$x_{1} - x_{2} \leq 29$
is implied by other inequalities in \eqref{L2ineqCyc} 
associated with simple mixed cycles, as follows.
The edge set of $C$ is the union of two
simple mixed cycles
$C\sp{(1)} = (2,3)_{1} (3,4)_{2} (4,2)_{2}$
and
$C\sp{(2)} =  (3,4)_{1} (4,1)_{1} (1,3)_{2}$.
The cycles
$C\sp{(1)}$ and $C\sp{(2)}$
have break vertices
$(i_{1}\sp{(1)}, j_{1}\sp{(1)})=(2,3)$
and $(i_{1}\sp{(2)}, j_{1}\sp{(2)})=(3,1)$, 
respectively, and give rise to
\begin{align*}
& x_{3} - x_{2} \leq \gamma(C\sp{(1)}) = 5+ (2+5) = 12,
\\ &
x_{1} - x_{3} \leq \gamma(C\sp{(2)}) = (8+7) + 2 = 17.
\end{align*}
These two inequalities, when added, imply the inequality
$x_{1} - x_{2} \leq 29$ corresponding to 
the non-simple mixed cycle $C$.
It is generally true that the inequality \eqref{L2ineqCyc} 
for a non-simple mixed cycle 
is implied by other inequalities in \eqref{L2ineqCyc} 
associated with simple mixed cycles.
\finbox
\end{itemize}
\end{example}

\begin{remark} \rm \label{RMwhyredundant}
In Example~\ref{EXl2graphcyc},
we have seen a structural reason (non-simplicity)
for redundancy.
Another (obvious) reason for redundancy
is numerical coincidence.
Such numerical redundancy is likely to occur, for example,
if $\gamma\sp{(1)}$ and $\gamma\sp{(2)}$
in the given descriptions in \eqref{PkdescG} 
satisfy triangle inequality.
There may be other reasons that lead to redundancy in 
inequalities in \eqref{L2ineqCyc}.
It is left for the future to clarify the condition for redundant inequalities.
\finbox
\end{remark}

\subsection{Refined theorems}
\label{SCthmrefined}

With the terminology and notation 
introduced in Section~\ref{SCgraphrep}, we can 
state the following theorem,
which is a refinement of Theorem~\ref{THl2polydesc}.
In \eqref{L2SdescCyc} as well as in \eqref{L2PdescCyc},
$( i_{1}, j_{1}, \ldots, i_{m}, j_{m})$ denotes the break vertices of a 
simple mixed cycle $C$
in $G_{*} = G_{1} + G_{2}\sp{\circ}$
and $\gamma(C)$ is the length of $C$ defined in \eqref{gammaCGdef}.

\begin{theorem} \label{THl2polydescCyc}
\quad

\noindent
{\rm (1)}
An \LL-convex set $S \subseteq \ZZ\sp{n}$ represented as
$S = S_{1} + S_{2}$ with $S_{1}$ and $S_{2}$ being L-convex
can be described as 
\begin{align}
S = {} & \{ x \in  \ZZ\sp{n} \mid  
 x(\{ j_{1}, \ldots, j_{m} \}) - x(\{ i_{1}, \ldots, i_{m} \}) \leq 
 \gamma(C),
\nonumber \\ & \phantom{\{ x \in  \ZZ\sp{n} \mid} \  
\mbox{\rm $C$: simple mixed cycle in $G_{*}$} \}.
 \label{L2SdescCyc}
\end{align}

\noindent
{\rm (2)}
An \LL-convex polyhedron 
$P \subseteq \RR\sp{n}$
represented as
$P = P_{1} + P_{2}$
with $P_{1}$ and $P_{2}$ being L-convex
can be described as 
\begin{align}
P = {} &  \{ x \in  \RR\sp{n} \mid  
 x(\{ j_{1}, \ldots, j_{m} \}) - x(\{ i_{1}, \ldots, i_{m} \}) \leq 
 \gamma(C),
\nonumber \\ & \phantom{\{ x \in  \RR\sp{n} \mid} \  
\mbox{\rm $C$: simple mixed cycle in $G_{*}$} \}.
 \label{L2PdescCyc}
\end{align}
\end{theorem}
\begin{proof}
Two different proofs are given in this paper.
The first is a structural proof, which 
is a refinement of the proof (Section~\ref{SCpolydescprf}) of Theorem~\ref{THl2polydesc} 
and relies on more detailed versions of 
the conjugacy theorem and \MM-optimality criterion
in discrete convex analysis.
The second is a direct algebraic proof,
which is based on Fourier--Motzkin elimination
and does not use results from discrete convex analysis.
The structural proof is given in Section~\ref{SCproofCyc}
and the algebraic proof in Section~\ref{SCproofL2fourmotz}.
\qedJIAM
\end{proof}

\begin{remark} \rm \label{RMthmcomp}
Theorem~\ref{THl2polydescCyc} is a refinement of Theorem~\ref{THl2polydesc}
in the following respects.
\begin{enumerate}

\item
Theorem~\ref{THl2polydesc} identified the form 
$x(J) - x(I) \leq \gamma_{IJ}$
of the inequalities,
but the bound $\gamma_{IJ}$ 
is not investigated.
In contrast,
Theorem~\ref{THl2polydescCyc} gives an expression 
of this $\gamma_{IJ}$ in terms of the 
given data of the constituent L-convex sets (or polyhedra).
In particular, triangle inequality is not assumed,
which is natural and convenient in applications.

\item
Theorem~\ref{THl2polydesc} involves 
all pairs $(I,J)$ of disjoint subsets $I$ and $J$
in \eqref{L2Sdesc} (or \eqref{L2Pdesc}),
whereas the expression 
\eqref{L2SdescCyc} (or \eqref{L2PdescCyc})
in Theorem~\ref{THl2polydescCyc} 
restricts itself to 
those pairs $(I,J)$ which 
are derived from the given descriptions of the constituent L-convex 
sets (or polyhedra).
\finbox
\end{enumerate}
\end{remark}

Theorem~\ref{THl2polydescCyc} is demonstrated below for 
the \LL-convex set in Example \ref{EXl2dim4polyG}.

\begin{example} \rm \label{EXl2desccyc}
Recall the \LL-convex set $S = S_{1} + S_{2}$ in Example \ref{EXl2dim4polyG},
for which the associated graphs are shown in Fig.~\ref{FGL2graph4}.
In Example~\ref{EXl2graphcyc}
we have seen typical cases of the inequality in \eqref{L2ineqCyc}
associated with a mixed cycle in $G_{*}$.
By inspecting all simple mixed cycles
in $G_{*}$, we arrive at the following
system of inequalities to describe $S$:
\begin{align} 
&  
 \phantom{x_{1} - x_{1} \leq 0, \quad\ }
 \phantom{x_{1} - x_{2} \leq 99, \quad\ }
 x_{1} - x_{3} \leq 17, \quad
 x_{1} - x_{4} \leq 11, \quad
\nonumber
\\ &
 x_{2} - x_{1} \leq 9, \quad
 \phantom{x_{2} - x_{2} \leq 0, \quad} \ \ \ 
 x_{2} - x_{3} \leq 21, \quad
 x_{2} - x_{4} \leq 15, \quad
\nonumber
\\ &
 x_{3} - x_{1} \leq 11, \quad
 x_{3} - x_{2} \leq 12, \quad
 \phantom{x_{3} - x_{3} \leq 0, \quad} \ 
 x_{3} - x_{4} \leq 17, \quad
\label{l2dim4polyGx}
\\ &
 x_{4} - x_{1} \leq 17, \quad
 x_{4} - x_{2} \leq 18, \quad
 x_{4} - x_{3} \leq 11, \quad
\nonumber
\\ &
  x_{2} + x_{4} - x_{1} - x_{3}  \leq 15,
\nonumber
\end{align}
where redundant inequalities are omitted.
In Section~\ref{SCelimappr} 
we show an alternative method to derive these inequalities.
\finbox
\end{example}

\begin{example} \rm \label{EXlongineqG}
For the \LL-convex set $S = S_{1} + S_{2}$ 
defined by \eqref{longineqS1} and \eqref{longineqS2}
in Example~\ref{EXlongineq},
the graph
$G_{1} + G_{2}\sp{\circ}$
is a cycle
$( 1, 2, \ldots, m-1, m)$.
There is only one simple mixed cycle $C$, for which
$( i_{1}, j_{1}, \ldots, i_{m}, j_{m})
=( 1, 2, \ldots, m-1, m)$.
This cycle corresponds to the inequality
$(x_{1} + x_{3} + \cdots + x_{n-1}) - (x_{2} + x_{4} + \cdots + x_{n}) \leq 0$
in \eqref{longineqS}.
\finbox
\end{example}

By the simple relation between \LL-convexity and \LLnat-convexity 
(Proposition \ref{PRfromL2toL2nat}),
Theorem~\ref{THl2polydescCyc} above
can be adapted easily to an \LLnat-convex set (or polyhedron).
Let $S = S_{1} + S_{2}\subseteq \ZZ\sp{n}$
be an \LLnat-convex set,
with $S_{1}$ and $S_{2}$ being \Lnat-convex.
Each $S_{k}$ is described as in Theorem~\ref{THlnatpolydesc}
with 
$\tilde \gamma_{ij}\sp{(k)}$ 
in \eqref{tildeGamma}.
Consider a graph 
$\tilde G_{k} =(\tilde N, \tilde E_{k})$
with $\tilde N = N \cup \{ 0 \}$ and
$\tilde E_{k} = \{ (i,j) \mid 
  \tilde \gamma_{ij}\sp{(k)} < +\infty, \ i,j \in \tilde N, \ i \neq j \} $.
Let $\tilde G_{2}\sp{\circ}$ be the reorientation of $\tilde G_{2}$, and 
let $\tilde{G}_{*} = \tilde G_{1} + \tilde G_{2}\sp{\circ}$.
The notion of a mixed cycle $C$
can be defined naturally on $\tilde{G}_{*}$,
and $\tilde \gamma(C)$ will denote 
the length of $C$ defined similarly to \eqref{gammaCGdef}
using $\tilde \gamma_{ij}\sp{(k)}$.

\begin{theorem} \label{THlnat2polydescCyc}
\quad

\noindent
{\rm (1)}
An \LLnat-convex set 
$S \subseteq \ZZ\sp{n}$
represented as
$S = S_{1} + S_{2}$
with $S_{1}$ and $S_{2}$ being \Lnat-convex
can be described as 
\begin{align}
S = {} & \{ x \in \ZZ\sp{n} \mid 
 \tilde x = (x,0) \in \ZZ\sp{n+1}, \ 
 \tilde x(\{ j_{1}, \ldots, j_{m} \}) - \tilde x(\{ i_{1}, \ldots, i_{m} \}) \leq 
 \tilde \gamma(C),
\nonumber \\ & \phantom{\{ x \in  \ZZ\sp{n} \mid} \  
\mbox{\rm $C$: simple mixed cycle in $\tilde G_{*}$} \}.
 \label{L2natSdescCyc}
\end{align}

\noindent
{\rm (2)}
An \LLnat-convex polyhedron $P \subseteq \RR\sp{n}$
represented as
$P = P_{1} + P_{2}$
with $P_{1}$ and $P_{2}$ being \Lnat-convex
can be described as 
\begin{align}
P = {} &  \{ x \in  \RR\sp{n} \mid  
 \tilde x = (x,0)  \in \RR\sp{n+1}, \ 
 \tilde x(\{ j_{1}, \ldots, j_{m} \}) - \tilde x(\{ i_{1}, \ldots, i_{m} \}) \leq 
 \tilde \gamma(C),
\nonumber \\ & \phantom{\{ x \in  \RR\sp{n} \mid} \  
\mbox{\rm $C$: simple mixed cycle in $\tilde G_{*}$} \}.
 \label{L2natPdescCyc}
\end{align}
\end{theorem}
\begin{proof}
This follows from Theorem~\ref{THl2polydescCyc} 
with Propositions \ref{PRminkowrestr} and \ref{PRfromL2toL2nat}.
\qedJIAM
\end{proof}

\subsection{Proof of Theorem~\ref{THl2polydescCyc} by discrete convex analysis}
\label{SCproofCyc}

We give a proof of Theorem~\ref{THl2polydescCyc}
relying on results from discrete convex analysis,
which is similar in vein to the proof of
Theorem~\ref{THl2polydesc}
but uses detailed versions of 
the conjugacy theorem and \MM-optimality criterion.
To be specific, 
while we used the \MM-optimality criterion 
of \cite[Theorem~8.32]{Mdcasiam}
to prove Theorem~\ref{THl2polydesc},
the proof of Theorem~\ref{THl2polydescCyc} here 
is based on another form of \MM-optimality criterion 
\cite[Theorem~8.33]{Mdcasiam}
that is applicable when an \MM-convex function is represented as 
the sum of two M-convex functions.
We shall use concepts 
such as M- and \MM-convex functions
as well as L- and \LL-convex functions.
For definitions of these concepts,
the reader is referred to Appendix \ref{SClmfn}.

\subsubsection{Step~1 (using DCA structural results)}
\label{SCproofCycDCA}

Let $S \subseteq \ZZ\sp{n}$ be an \LL-convex set,
represented as
$S = S_{1} + S_{2}$ 
with {\rm L}-convex sets
$S_{1}$ and $S_{2}$.
Denote the indicator functions of 
$S$ and $S_{k}$
by 
$g$ and $g_{k}$,
respectively; that is,
$g = \delta_{S}$ and
$g_{k} = \delta_{S_{k}}$ for $k=1,2$.
Since $S_{k}$ is an L-convex set,
each $g_{k}: \ZZ\sp{n} \to  \{ 0, +\infty \}$ 
is an L-convex function. 
By $S = S_{1} + S_{2}$, the function
$g$ is equal to the (integral) infimal convolution 
$g_{1} \Box g_{2}$ of $g_{1}$ and $g_{2}$, that is,
\[
g(x) =  (g_{1} \Box g_{2})(x) = \inf\{ g_{1}(y)+ g_{2}(z) 
        \mid x = y+z; \  y,z \in \ZZ\sp{n}  \},
\]
which shows that
$g: \ZZ\sp{n} \to  \{ 0, +\infty \}$ 
is an \LL-convex function.

Let $f_{k}$ denote the conjugate of the function $g_{k}$,
that is,
$f_{k} = g_{k}\sp{\bullet}$ for $k=1,2$.
By the conjugacy theorem 
\cite[Theorem~8.12]{Mdcasiam},
each $f_{k}: \ZZ\sp{n} \to \ZZ \cup \{ +\infty \}$ 
is an {\rm M}-convex function, 
and $g_{k} = f_{k}\sp{\bullet}$.
In addition, 
$f_{k}(\bm{0})=0$ and
$f_{k}$ is positively homogeneous, since it is the conjugate of an indicator function.

Define $f= f_{1} + f_{2}$,
which is an \MM-convex function.
We have
\begin{equation} \label{M2fL2gconj}
 f = f_{1} + f_{2}
= g_{1}\sp{\bullet} + g_{2}\sp{\bullet} 
= (g_{1} \Box g_{2})\sp{\bullet} 
=  g\sp{\bullet} ,
\end{equation}
where the equality 
$g_{1}\sp{\bullet} + g_{2}\sp{\bullet} = (g_{1} \Box g_{2})\sp{\bullet}$
is in \cite[p.~229, (8.38)]{Mdcasiam}.
We also have
\begin{equation} \label{M2fconjL2g}
 f\sp{\bullet} = 
(f_{1} + f_{2})\sp{\bullet}
= f_{1}\sp{\bullet} \Box  f_{2}\sp{\bullet} 
= g_{1} \Box g_{2}
=  g ,
\end{equation}
where the equality
$(f_{1} + f_{2})\sp{\bullet}= f_{1}\sp{\bullet} \Box  f_{2}\sp{\bullet}$ 
is due to \cite[Theorem~8.36]{Mdcasiam}.
Thus we obtain $f =  g\sp{\bullet}$ and $f\sp{\bullet}=  g$,
which allows us to  
use the expression 
$S =  \subgZ f(\bm{0})$
in \eqref{M2SisSubgf} with $f=f_{1}+ f_{2}$.
Therefore,
\begin{align} 
 S =  \subgZ f(\bm{0})
&
 = \{ p \in  \ZZ\sp{n}  \mid 
    \bm{0} \in \argmin_{z} \{ f(z) -  \langle p, z \rangle \} \ \}
\nonumber \\ &
 = \{ p \in  \ZZ\sp{n}  \mid 
    \bm{0} \in \argmin_{z} \{ (f_{1}(z) -  \langle p, z \rangle) + f_{2}(z) \} \ \}.
\label{M2Sargmin}
\end{align}

In \eqref{M2Sargmin}, the functions
$f_{1}(z) -   \langle p, z \rangle$ and $f_{2}(z)$
are M-convex.
For a general \MM-convex function
$h: \ZZ^{n} \to \RR \cup \{ +\infty  \}$
represented as $h= h_{1}+h_{2}$
with M-convex functions $h_{1}$ and $h_{2}$,
the second form of \MM-optimality criterion 
\cite[Theorem~8.33]{Mdcasiam}
states that a vector $z^{*} \in \ZZ\sp{n}$
with $h(z^{*}) < +\infty$
is a minimizer of $h$
if and only if
\begin{align}  
& \sum_{r=1}\sp{m} 
  [h_{1}(z^{*} -\unitvec{i_{r}} + \unitvec{j_{r}}) - h_{1}(z^{*})]
\nonumber \\ &
 {} + \sum_{r=1}\sp{m} 
  [h_{2}(z^{*} - \unitvec{i_{r+1}} + \unitvec{j_{r}}) - h_{2}(z^{*})]
   \geq 0 
\label{M2optCycle}
\end{align}
for any distinct 
$i_{1}, j_{1}, \ldots, i_{m}, j_{m}\in N$,
where $i_{m+1}=i_{1}$ by convention.%
\footnote{%%%%%%%%%%
The statement of \cite[Theorem~8.33]{Mdcasiam} imposes the condition 
``$\{ i_{1}, \ldots, i_{m} \} \cap \{ j_{1}, \ldots, j_{m} \} =\emptyset$''
which allows the possibility of 
$i_{p} = i_{q}$ or $j_{p} = j_{q}$ for $p \neq q$,
but this can be strengthened to the condition that 
$i_{1}, \ldots, i_{m}, j_{1}, \ldots, j_{m}$ should be 
all distinct. There is a typo in \cite[page 228, line 5]{Mdcasiam}:
``$f_{2}(x + \chi_{u_{i+1}} - \chi_{v_{i}})$'' 
should be
``$f_{2}(x - \chi_{u_{i+1}} + \chi_{v_{i}})$.'' 
}  %%%footnote%%%% 
The condition \eqref{M2optCycle}
for $h_{1}(z) = f_{1}(z) -   \langle p, z \rangle$,
$h_{2}(z) = f_{2}(z)$, and $z^{*} =\bm{0}$
reads
\begin{equation}   \label{subgM2fncycIneq}
 \sum_{r=1}\sp{m} ( p_{j_{r}} - p_{i_{r}} ) 
\leq 
 \sum_{r=1}\sp{m} 
( f_{1}(\unitvec{j_{r}}-\unitvec{i_{r}}) 
  +  f_{2}(\unitvec{j_{r}} -\unitvec{i_{r+1}})  ),
\end{equation}
which is a necessary and sufficient condition for $p$ to be in $S$.
Therefore,
\begin{equation}
S =  \{ p \in  \ZZ\sp{n} \mid  
 \mbox{\eqref{subgM2fncycIneq} for all distinct 
   \  $i_{1}, j_{1}, \ldots, i_{m}, j_{m}$} \}.
\label{subgM2fncyc}
\end{equation}

On the right-hand side of the inequality \eqref{subgM2fncycIneq},
we observe
\begin{align} 
 f_{k}(\unitvec{j}-\unitvec{i}) 
 & = \sup\{ \langle \unitvec{j}-\unitvec{i},x \rangle \mid x \in S_{k} \}
\nonumber \\
  & = \sup\{  x_{j} - x_{i}  \mid x \in S_{k} \}
 =\lambda(i,j;G_{k}),
 \label{M2prooffkxxlambda}
\end{align}
where 
the first equality is due to the definition of 
$f_{k}= \delta_{S_{k}}\sp{\bullet}$
and the last equality is 
a fundamental relation between 
the maximum potential difference
and the shortest path length;
see, e.g., \cite[Theorem~8.3]{Sch03}.
(Recall from Section~\ref{SCgraphrep} that
$\lambda(i,j;G_{k})$ denotes 
the minimum $\gamma\sp{(k)}$-length of a path connecting $i$ to $j$ in $G_{k}$.)
With the use of \eqref{M2prooffkxxlambda}
we can rewrite the right-hand side of \eqref{subgM2fncycIneq} as
\[
 \sum_{r=1}\sp{m} 
( f_{1}(\unitvec{j_{r}}-\unitvec{i_{r}}) 
  +  f_{2}(\unitvec{j_{r}} -\unitvec{i_{r+1}})  )
=  \sum_{r=1}\sp{m} (
 \lambda(i_{r},j_{r};G_{1}) +
 \lambda(j_{r},i_{r+1};G_{2}\sp{\circ}) ).
\]
By introducing notation
\begin{equation}   \label{lambdadef}
 \lambda( i_{1}, j_{1}, \ldots, i_{m}, j_{m}) :=
 \sum_{r=1}\sp{m} (
 \lambda(i_{r},j_{r};G_{1}) +
 \lambda(j_{r},i_{r+1};G_{2}\sp{\circ}) )
\end{equation}
and changing the variable $p$ to $x$,
we can rewrite the inequality in 
\eqref{subgM2fncycIneq} as
\begin{equation}   \label{pJpIcycG2}
 x(\{ j_{1}, \ldots, j_{m} \}) - x(\{ i_{1}, \ldots, i_{m} \}) \leq 
 \lambda( i_{1}, j_{1}, \ldots, i_{m}, j_{m})
\end{equation}
and the representation of $S$ in \eqref{subgM2fncyc} as
\begin{equation}
S =  \{ x \in  \ZZ\sp{n} \mid  
\mbox{\eqref{pJpIcycG2}
 for all distinct \  
$i_{1}, j_{1}, \ldots, i_{m}, j_{m}$} \}.
 \label{subgM2fncyc2}
\end{equation}

\subsubsection{Step~2 (using cycle decomposition)}
\label{SCproofCycG}

The next step of the proof is to relate \eqref{subgM2fncyc2}
to simple mixed cycles in $G_{*} = G_{1} + G_{2}\sp{\circ}$.
Let $C$ be a (simple or non-simple)  mixed cycle
with break vertices
$( i_{1}, j_{1}, \ldots, i_{m}, j_{m})$,
and consider the inequality
\begin{equation} \label{L2ineqCycagain}
 x(\{ j_{1}, \ldots, j_{m} \}) - x(\{ i_{1}, \ldots, i_{m} \}) \leq 
 \gamma(C)
\end{equation}
in \eqref{L2ineqCyc},
where $\gamma(C)$ denotes the length of the cycle $C$
defined in \eqref{gammaCGdef}. 
Using these inequalities for simple mixed cycles $C$,
we define $\hat S \subseteq \ZZ\sp{n}$ by
\begin{equation}
\hat S :=  \{ x \in  \ZZ\sp{n} \mid  
\mbox{\rm \eqref{L2ineqCycagain}
 for every simple mixed cycle $C$} \}.
 \label{L2SdescCycPrf}
\end{equation}
We want to show that $S = \hat S$,
which is \eqref{L2SdescCyc}
in Theorem~\ref{THl2polydescCyc}.

The inclusion $S \subseteq \hat S$ is easy to see.
Let 
$C = e_{1} e_{2} \cdots e_{\ell-1} e_{\ell}$
be a simple mixed cycle with break vertices
$(i_{1}, \ldots, i_{m}, j_{1}, \ldots, j_{m})$.
Since
$\lambda(i_{r},j_{r};G_{1})$ 
and $\lambda(j_{r}, i_{r+1};G_{2}\sp{\circ})$
denote shortest path lengths, we have
\begin{align*}
 \lambda( i_{1}, j_{1}, \ldots, i_{m}, j_{m}) & =
 \sum_{r=1}\sp{m} (
 \lambda(i_{r},j_{r};G_{1}) +
 \lambda(j_{r},i_{r+1};G_{2}\sp{\circ}) )
\\ &
\leq  \sum_{k=1}\sp{\ell} \gamma(e_{k}) = \gamma(C),
\end{align*}
which implies $S \subseteq \hat S$.

The reverse inclusion $S \supseteq \hat S$ 
can be shown as follows.
Let 
$( i_{1}, j_{1}, \ldots, \allowbreak i_{m}, j_{m})$
be an arbitrary tuple of distinct indices
with $\lambda( i_{1}, j_{1}, \ldots, i_{m}, j_{m}) < +\infty$.
We will show that there exists a family 
of simple mixed cycles
such that 
the inequalities \eqref{L2ineqCycagain}
for this family
imply the inequality  
\eqref{pJpIcycG2} for 
$( i_{1}, j_{1}, \ldots, \allowbreak i_{m}, j_{m})$.
Then the inclusion $S \supseteq \hat S$ follows.

Since $\lambda( i_{1}, j_{1}, \ldots, i_{m}, j_{m})$ is finite, 
it follows from the definition 
in \eqref{lambdadef} that
$\lambda(i_{r},j_{r};G_{1}) < +\infty$ and 
$\lambda(j_{r},i_{r+1};G_{2}\sp{\circ}) < +\infty$ 
for $r=1,2,\ldots, m$.
Let $L_{r}\sp{(1)}$ be a shortest path from $i_{r}$ to $j_{r}$ in $G_{1}$
with minimum number of edges.
Similarly,
let $L_{r}\sp{(2)}$ 
be a shortest path from $j_{r}$ to $i_{r+1}$ in $G_{2}\sp{\circ}$
with minimum number of edges.
We then have
\begin{equation}  \label{gammaLrlamda}
\gamma(L_{r}\sp{(1)}) = \lambda(i_{r},j_{r};G_{1}),
\  
\gamma(L_{r}\sp{(2)}) = \lambda(j_{r},i_{r+1};G_{2}\sp{\circ})
\quad 
(r=1,2,\ldots, m).
\end{equation}
The concatenation (series connection) of
$L_{1}\sp{(1)},  L_{1}\sp{(2)}, 
L_{2}\sp{(1)},  L_{2}\sp{(2)}, \ldots, 
L_{m}\sp{(1)},  L_{m}\sp{(2)}$
determines a cycle $\tilde C$
in $G_{*}$. 
This cycle $\tilde C$ is not necessarily simple.
(The paths
$L_{1}\sp{(1)},  L_{1}\sp{(2)}, 
L_{2}\sp{(1)},  L_{2}\sp{(2)}, \ldots, 
L_{m}\sp{(1)},  L_{m}\sp{(2)}$
may possibly have common edges,
and in such a case, it will be more precise to call 
$\tilde C$ a closed walk, although we refer to it as a cycle.)

We can decompose $\tilde C$ into a family of simple cycles, say, 
$\{ C_{q} \mid q \in K \}$.
By \eqref{gammaLrlamda} as well as \eqref{lambdadef} we have
\begin{align} 
& \sum_{q \in K} \gamma(C_{q}) 
= \sum_{r=1}\sp{m} ( \gamma(L_{r}\sp{(1)}) + \gamma(L_{r}\sp{(2)}) )
\nonumber \\
& = \sum_{r=1}\sp{m} (\lambda(i_{r},j_{r};G_{1}) + \lambda(j_{r},i_{r+1};G_{2}\sp{\circ}))
= \lambda( i_{1}, j_{1}, \ldots, i_{m}, j_{m}).
 \label{gammaCqHatlamda}
\end{align}
Each cycle $C_{q}$ is simple but may or may not be mixed.
If $C_{q}$ is mixed, 
we can think of an inequality \eqref{L2ineqCycagain}
associated with $C_{q}$, 
which we express as
\begin{equation} \label{L2ineqCycagainCq}
 x(J_{q}) - x(I_{q}) \leq  \gamma(C_{q}),
\end{equation}
where $I_{q} \cup J_{q}$ (with an appropriate ordering of elements)
is the break vertices of $C_{q}$.
If $C_{q}$ is not mixed, 
we have
$\gamma(C_{q}) \geq 0$ because 
neither $G_{1}$ nor $G_{2}\sp{\circ}$
contains negative cycles,
and hence the inequality \eqref{L2ineqCycagainCq} 
is also true under the definition of $I_{q} = J_{q} = \emptyset$.

A crucial observation here is 
the following counting relation.
For $i \in N$ and $I \subseteq N$ we define 
\begin{equation} \label{memberindicator1}
\varepsilon(i, I) = 
\begin{cases}
  1 & \text{($i \in I$)}, \\
  0 & \text{($i \notin I$)}. 
\end{cases}
\end{equation}

\begin{lemma} \label{LMcntCyc}
For each $i \in N$, we have
\begin{align} & 
 \sum_{q \in K} ( \varepsilon(i, J_{q})
  -  \varepsilon(i, I_{q}) )
 = 
\begin{cases}
 +1 & (i \in \{ j_{1}, j_{2}, \ldots, j_{m} \}), \\
 -1 & (i \in \{ i_{1}, i_{2}, \ldots, i_{m} \}), \\
 0 & (\text{\rm otherwise}). \\
\end{cases}
\label{BkVertCq}
\end{align}
\end{lemma}
\begin{proof}
Let $E_{\rm in}(i)$ and $E_{\rm out}(i)$ denote the (multi)sets of the edges of 
$\tilde C$ that enter and leave $i$, respectively.
First we consider the case of $i=j_{r}$, where $1 \leq r \leq m$.
The last edge of $L_{r}\sp{(1)}$ ($\subseteq E_{1}$) enters $j_{r}$ 
and the first edge of $L_{r}\sp{(2)}$ ($\subseteq E_{2}\sp{\circ}$) leaves $j_{r}$.
The vertex $j_{r}$ may be contained in the middle other paths $L_{s}\sp{(k)}$
with $s \neq r$, but in this case, the two consecutive edges connected at $j_{r}$ 
on $L_{s}\sp{(k)}$ belong to the same class ($E_{1}$ or $E_{2}\sp{\circ}$).
Therefore, we have
\begin{equation} \label{jrEinEout}
|E_{1} \cap E_{\rm in}(j_{r})| = |E_{1} \cap E_{\rm out}(j_{r})|+1,
\ \ 
|E_{2}\sp{\circ} \cap E_{\rm in}(j_{r})| = |E_{2}\sp{\circ} \cap E_{\rm out}(j_{r})|-1.
\end{equation}
Suppose that $C_{q}$ passes through $j_{r}$,
and let $e_{\rm in} \in E_{\rm in}(j_{r})$ and 
$e_{\rm out} \in E_{\rm out}(j_{r})$ be the edges of $C_{q}$ 
that enter and leave $j_{r}$, respectively.
If $e_{\rm in} \in E_{1}$
and $e_{\rm out} \in E_{2}\sp{\circ}$, then $j_{r} \in J_{q}$
(i.e., $\varepsilon(j_{r}, J_{q}) = 1$).
Symmetrically, 
if $e_{\rm in} \in E_{2}\sp{\circ}$
and $e_{\rm out} \in E_{1}$, then $j_{r} \in I_{q}$.
If
$\{ e_{\rm in}, e_{\rm out} \} \subseteq E_{1}$
or
$\{ e_{\rm in}, e_{\rm out} \} \subseteq E_{2}\sp{\circ}$,
then
$j_{r} \notin I_{q} \cup J_{q}$.
The vertex $j_{r}$ may be contained in several $C_{q}$, but 
it follows from \eqref{jrEinEout} that
$ \sum_{q \in K} \varepsilon(j_{r}, J_{q})
  -  \sum_{q \in K} \varepsilon(j_{r}, I_{q}) = +1$,
as in \eqref{BkVertCq}.
The case of $i=i_{r}$ can be treated in a similar manner with
\begin{equation*} %%\label{irEinEout}
|E_{1} \cap E_{\rm in}(i_{r})| = |E_{1} \cap E_{\rm out}(i_{r})|-1,
\quad 
|E_{2}\sp{\circ} \cap E_{\rm in}(i_{r})| = |E_{2}\sp{\circ} \cap E_{\rm out}(i_{r})|+1.
\end{equation*}
Also the remaining case of $i \neq i_{r}, j_{r}$ can treated similarly using
\begin{equation*} %%\label{otherEinEout}
|E_{1} \cap E_{\rm in}(i)| = |E_{1} \cap E_{\rm out}(i)|,
\quad 
|E_{2}\sp{\circ} \cap E_{\rm in}(i)| = |E_{2}\sp{\circ} \cap E_{\rm out}(i)|.
\end{equation*}
Thus \eqref{BkVertCq} is proved.
\qedJIAM
\end{proof}

The addition of \eqref{L2ineqCycagainCq}
over $q \in K$ gives
\begin{equation} \label{L2ineqCycagainCqsum}
\sum_{q \in K} (  x(J_{q}) - x(I_{q}) ) \leq \sum_{q \in K}  \gamma(C_{q}).
\end{equation}
For the left-hand side of \eqref{L2ineqCycagainCqsum},
we have
\[
\sum_{q \in K} ( x(J_{q}) - x(I_{q}) ) 
 = x(\{ j_{1}, \ldots, j_{m} \}) - x(\{ i_{1}, \ldots, i_{m} \}) 
\]
by Lemma~\ref{LMcntCyc},
while the right-hand side of \eqref{L2ineqCycagainCqsum}
is equal to 
$\lambda( i_{1}, j_{1}, \ldots, i_{m}, j_{m})$ by \eqref{gammaCqHatlamda}.
This show that each inequality in \eqref{pJpIcycG2}
can be derived from some of the inequalities in \eqref{L2ineqCycagain}.
From this we can conclude that 
$S \supseteq \hat S$,
completing the proof of Theorem~\ref{THl2polydescCyc}.

\begin{remark} \rm \label{RMsuppfn}
The {\em supporting function} of a polyhedron $Q \subseteq \RR\sp{n}$
(in general) 
is defined for all $u \in \RR\sp{n}$ by
\begin{equation*}  %% \label{supfndef}
\eta(Q,u) =  \sup \{ \langle u, x \rangle \mid x \in Q \} .
\end{equation*}
Then $Q$ is described by a system of inequalities
$\langle u, x \rangle  \leq \eta(Q,u)$
with a suitable finite set of $u$'s.
If $Q$ is a full-dimensional bounded polyhedron, the vectors $u$
will be the normal vectors of all facets of $Q$. 
It is known \cite[Section 2.2, Exercise 8]{Gru03}
that the supporting function of a Minkowski sum 
is given by the sum of the respective supporting functions:
\begin{equation*}  %%  \label{supfnMinkow}
\eta(Q_{1}+ Q_{2},u) = \eta(Q_{1},u) + \eta(Q_{2},u). 
\end{equation*}
This shows that
the bounding constant $\gamma_{IJ}$
in \eqref{L2Pdesc} 
for \LL-convex polyhedron
$P = P_{1}+P_{2}$ 
is given as
\begin{equation}  \label{gammaL2}
 \gamma_{IJ} = \eta(P,\unitvec{J} - \unitvec{I}) 
  = \eta(P_{1},\unitvec{J} - \unitvec{I}) + \eta(P_{2},\unitvec{J} - \unitvec{I}). 
\end{equation}
Let $S = S_{1}+S_{2}$ be an \LL-convex set.
Since $f = \delta_{S}\sp{\bullet}$ is given by 
\[
f (u) 
= \sup\{  \langle u, x \rangle - \delta_{S}(x)  \mid x \in \ZZ\sp{n} \}
= \sup\{  \langle u, x \rangle  \mid x \in S \}
\qquad ( u \in \ZZ\sp{n}),
\]
the function $f$ is nothing but the supporting function $\eta(P,\cdot)$ 
of $P=\overline{S}$ restricted to integral vectors.
Similarly, $f_{k}$ is essentially 
the same as the supporting function
$\eta(P_{k},\cdot)$ of $P_{k}=\overline{S_{k}}$.
Furthermore, the relation
$f= f_{1} + f_{2}$
evaluated at $u = \unitvec{J} - \unitvec{I}$
corresponds to 
\eqref{gammaL2}.
\finbox
\end{remark}

\section{Elimination approach to \LL-convex sets}
\label{SCelimappr}

In this section we give an algebraic proof of Theorem~\ref{THl2polydescCyc}
by means of the Fourier--Motzkin elimination.

\subsection{Fourier--Motzkin elimination}
\label{SCproofL2fourmotzmethod}

The procedure of Fourier--Motzkin elimination \cite{Sch86,Zie07} 
is described here for a (general) system of inequalities
\begin{equation}\label{ineqApb}
 A u \leq b
\end{equation}
in $u \in \RR\sp{n}$.
It is assumed that the matrix $A$ has entries from $\{ -1,0,+1 \}$,
which is the case with our system 
\eqref{QP1descPrf2}--\eqref{QP2descPrf2}
in Section~\ref{SCproofL2ineq}.
Let $R$ denote the row set of
$A$, that is,
$A = ( a_{ij} \mid i \in R, j \in \{ 1,2,\ldots, n \})$.
The $i$th row vector of $A$ is denoted by $a_{i}$ for $i \in R$.
By assumption, we have $a_{ij} \in \{ -1,0,+1 \}$
for all $i$ and $j$.

The Fourier--Motzkin elimination for (\ref{ineqApb}) goes as follows.
According to the value of coefficient $a_{i1}$ of the first variable $u_{1}$,
we partition $R$ into three disjoint parts $(R_{1}^{+},R_{1}^{-},R_{1}^{0})$ as
\begin{align*}
 R_{1}^{+} &= \{ i \in R \mid a_{i1} = +1 \},   
\\
 R_{1}^{-} &= \{ i \in R \mid a_{i1} = -1 \},
\\ 
 R_{1}^{0} &= \{ i \in R \mid a_{i1} = 0 \},
\end{align*}
and decompose (\ref{ineqApb}) into three parts as
\begin{align}
 a_{i}  u \leq b_{i} &
\qquad (i \in R_{1}^{+}),
\label{ineqApbi+}
\\
 a_{i}  u \leq b_{i} &
\qquad (i \in R_{1}^{-}) ,
\label{ineqApbi-}
\\
 a_{i}  u \leq b_{i} &
\qquad (i \in R_{1}^{0}) .
\label{ineqApbi0}
\end{align}
For all possible combinations of $i \in R_{1}^{+}$ and $k \in R_{1}^{-}$,
we add the inequality for $i$ in (\ref{ineqApbi+}) 
and the inequality for $k$ in (\ref{ineqApbi-}) 
to generate
\begin{equation}\label{FMp2pn}
 (a_{i} + a_{k}) u \leq b_{i} + b_{k} 
\qquad (i \in R_{1}^{+},\; k \in R_{1}^{-}) .
\end{equation}
Since $a_{i1} + a_{k1}= 0$ for all
$i \in R_{1}^{+}$ and $k \in R_{1}^{-}$,
the newly generated inequalities in  (\ref{FMp2pn}) are free from the variable $u_{1}$.
We have thus eliminated $u_{1}$ and obtained
a system of inequalities 
in $(u_{2},\ldots,u_{n})$ consisting of (\ref{ineqApbi0}) and (\ref{FMp2pn}).

For the variable $u_{1}$ we obtain
\begin{equation}\label{FMp1}
 \max_{k \in R_{1}^{-}} \left\{ \sum_{j=2}^n a_{kj}u_{j} - b_{k} \right\} 
 \leq u_{1} \leq  
 \min_{i \in R_{1}^{+}} \left\{ b_{i} - \sum_{j=2}^n a_{ij}u_{j}  \right\} 
\end{equation}
from (\ref{ineqApbi+}) and (\ref{ineqApbi-}).
Once $(u_{2},\ldots,u_{n})$ 
is found,
$u_{1}$ can easily be obtained from (\ref{FMp1}).
Note that the interval described by (\ref{FMp1}) is nonempty 
as long as $(u_{2},\ldots,u_{n})$ satisfies (\ref{FMp2pn}).
It is understood that
the maximum over the empty set is equal to $-\infty$ 
and the minimum over the empty set is equal to $+\infty$.

It is emphasized that the derived system of inequalities 
in $(u_{1}, u_{2},\ldots,u_{n})$ 
consisting of  (\ref{ineqApbi0}), (\ref{FMp2pn}), and (\ref{FMp1})
is in fact equivalent to the original system consisting of
(\ref{ineqApbi+}), (\ref{ineqApbi-}), and (\ref{ineqApbi0}).
In particular, $(u_{1}, u_{2},\ldots,u_{n})$ satisfies 
(\ref{ineqApbi+}), (\ref{ineqApbi-}), and (\ref{ineqApbi0})
if and only if $(u_{2},\ldots,u_{n})$ satisfies 
(\ref{ineqApbi0}) and  (\ref{FMp2pn}),
and $u_{1}$ satisfies (\ref{FMp1}).
In geometric terms, the projection of 
the polyhedron
$Q = \{ u \in \RR\sp{n} \mid  A u \leq b \}$
to the space of $(u_{2},u_{3},\ldots,u_n)$
is described  by (\ref{ineqApbi0}) and (\ref{FMp2pn}).

The Fourier--Motzkin method applies the above procedure 
recursively to eliminate variables $u_{1},u_{2}, \ldots,u_{n-1}$.
At the stage when 
the variables $u_{1},u_{2}, \ldots,   \allowbreak   u_{\ell-1}$
have been eliminated,
we obtain a system of inequalities 
to describe the projection of $Q$ to the space of 
$(u_{\ell},u_{\ell+1},\ldots, \allowbreak u_n)$.
At the end of the process, a single inequality in $u_{n}$ of the form (\ref{FMp1}) results.
Then we can determine $(u_{1}, u_{2},\ldots,u_{n})$ 
in the reverse order $u_{n}, u_{n-1},\ldots,u_{1}$.

\subsection{Proof of Theorem~\ref{THl2polydescCyc} 
by Fourier--Motzkin elimination}
\label{SCproofL2fourmotz}

We present the proof for an \LL-convex polyhedron,
from which the proof for an \LL-convex set
follows immediately (see Remark~\ref{RMproofL2set}).
We prefer to work with polyhedra
rather than discrete sets
because of the geometric flavor of the proof 
with an interpretation by projections.

\subsubsection{Inequality systems}
\label{SCproofL2ineq}

Let $P = P_{1} + P_{2}$
be an \LL-convex polyhedron,
where $P_{1}$ and $P_{2}$ are
L-convex polyhedra.
Define 
\begin{equation} \label{l2Qdef}
 Q = \{ (x,y,z) \in \RR\sp{3n} \mid 
     x = y + z, \  y \in P_{1}, \  z \in P_{2} \}.
\end{equation}
Then 
$P  = \{ x \in \RR\sp{n} \mid (x,y,z) \in Q \}$, 
which is a projection of $Q$.
By Theorem~\ref{THlpolydesc} 
we have
\begin{align} 
P_{1} &= \{ y \in \RR\sp{n} \mid 
  y_{j} - y_{i} \leq \gamma_{ij}\sp{(1)} \ \ ((i,j) \in E_{1}) \},
\label{P1descPrf} 
\\
P_{2} &= \{ z \in \RR\sp{n} \mid 
  z_{i} - z_{j} \leq \gamma_{ji}\sp{(2)} \ \ ((j,i) \in E_{2}) \},
\label{P2descPrf} 
\end{align}
where  $E_{1}, E_{2} \subseteq (N \times N) \setminus \{ (i,i) \mid i \in N \}$, 
$\gamma_{ij}\sp{(1)} \in \RR$ (finite-valued)
for all $(i,j) \in E_{1}$,
and $\gamma_{ji}\sp{(2)} \in \RR$ 
for all $(j,i) \in E_{2}$.
Then $Q$ is described by
\begin{align} 
&  x_{i} = y_{i} + z_{i}
 \quad (i \in N),
\label{QP1P2descPrf} 
\\ &
y_{j} - y_{i} \leq \gamma_{ij}\sp{(1)} 
 \quad ((i,j) \in E_{1}),
\label{QP1descPrf} 
\\ &
 z_{i} - z_{j} \leq \gamma_{ji}\sp{(2)} 
\quad ((j,i) \in E_{2}).
\label{QP2descPrf} 
\end{align}

We can eliminate $z$ by substituting
$z_{i} = x_{i}  - y_{i}$ into \eqref{QP2descPrf},  to obtain
\begin{align} 
& 
y_{j} - y_{i} \leq \gamma_{ij}\sp{(1)} 
 \quad ((i,j) \in E_{1}),
\label{QP1descPrf2} 
\\ &
 y_{j} - y_{i} + x_{i} - x_{j} \leq \gamma_{ji}\sp{(2)} 
\quad ((j,i) \in E_{2}).
\label{QP2descPrf2} 
\end{align}
Let $\hat Q$ denote the set of $(x,y)$ satisfying these inequalities, that is,
\begin{equation} \label{l2Qhat}
 \hat Q = \{ (x,y) \in \RR\sp{2n} \mid 
   \eqref{QP1descPrf2},    \eqref{QP2descPrf2}  \} .
\end{equation}
We denote by
$\hat Q_{\ell}$
the projection of $\hat Q$ to the space of
$(x_{1}, x_{2}, \ldots, x_{n}, 
\allowbreak y_{\ell}, y_{\ell+1}, \allowbreak \ldots, \allowbreak y_{n})$
for $\ell=1,2,\ldots,n+1$.
We have $\hat Q_{1} = \hat Q$
and $\hat Q_{n+1} = P$.
By eliminating $y$ from 
\eqref{QP1descPrf2} and \eqref{QP2descPrf2},
we can obtain the polyhedral description of $P$.
The Fourier--Motzkin elimination procedure enables us
to carry out this task.

\begin{example} \rm \label{EXl2dim4polyFM}
Consider the \LL-convex polyhedron $P$
associated with the \LL-convex set in Example~\ref{EXl2dim4polyG}.
The inequalities in \eqref{QP1P2descPrf}--\eqref{QP2descPrf} are given by
\begin{align*} 
&  
x_{1} = y_{1} + z_{1}, \quad 
x_{2} = y_{2} + z_{2}, \quad 
x_{3} = y_{3} + z_{3}, \quad 
x_{4} = y_{4} + z_{4},  
\nonumber
\\ &
 y_{2} - y_{1} \leq 3, \quad
 y_{3} - y_{2} \leq 5, \quad
 y_{4} - y_{3} \leq 8, \quad
 y_{1} - y_{4} \leq 7, 
%%\label{l2dim4polyFMxyz}
\\ &
 z_{1} - z_{3} \leq 2, \quad
 z_{4} - z_{1} \leq 1, \quad
 z_{2} - z_{3} \leq 3, \quad
 z_{4} - z_{2} \leq 5, \quad
 z_{3} - z_{4} \leq 2,
\nonumber
\end{align*}
and those in  
\eqref{QP1descPrf2} and \eqref{QP2descPrf2} are given by
\begin{align*} 
&  
 y_{2} - y_{1} \leq 3, \quad
 y_{3} - y_{2} \leq 5, \quad
 y_{4} - y_{3} \leq 8, \quad
 y_{1} - y_{4} \leq 7, 
\nonumber
\\ &
  y_{3} - y_{1} + x_{1} - x_{3}  \leq 2, \quad
  y_{1} - y_{4} + x_{4} - x_{1}   \leq 1, 
%%\label{l2dim4polyFMxy}
\\ &
  y_{3} - y_{2} + x_{2} - x_{3}  \leq 3, \quad
  y_{2} - y_{4} + x_{4} - x_{2}  \leq 5, \quad
  y_{4} - y_{3} + x_{3} - x_{4}  \leq 2.
\nonumber
\end{align*}
By eliminating $(y_{1},y_{2}, y_{3}, y_{4})$, 
we arrive at a system of inequalities for $P$,
which is given by \eqref{l2dim4polyGx} in Example~\ref{EXl2desccyc}.
\finbox
\end{example}

\begin{remark} \rm \label{RMproofL2set}
Theorem~\ref{THl2polydescCyc}(1) for an \LL-convex set
follows from the statement (2) for an \LL-convex polyhedron. 
Let $S=S_{1} + S_{2}$,
where
$S_{1}$ and $S_{2}$ are L-convex sets, 
and denote their convex hulls by
$P = \overline{S}$, $P_{1} = \overline{S_{1}}$, and
$P_{2} = \overline{S_{2}}$.
Since 
$\overline{S_{1}+ S_{2}} = \overline{S_{1}} + \overline{S_{2}}$
in general
(\cite[Proposition 3.17]{Mdcasiam}),
we have $P = P_{1}+P_{2}$, where
$P_{1}$ and $P_{2}$ are L-convex polyhedra.
We have
$S = P \cap \ZZ\sp{n}$
by integral convexity of an \LL-convex set (\cite[Theorem~8.42]{Mdcasiam}),
and hence \eqref{L2PdescCyc} implies \eqref{L2SdescCyc}.
\finbox
\end{remark}

\subsubsection{Polyhedral description of the projection $\hat Q_{\ell}$}

In Section~\ref{SCgraphrep} we have 
introduced an inequality 
\begin{align}
x(\{ j_{1}, \ldots, j_{m} \}) - x(\{ i_{1}, \ldots, i_{m} \}) 
 & \leq  \gamma(C)
 \label{L2ineqCycFM}
\end{align}
associated with a mixed cycle $C$
with reference to its break vertices
(see \eqref{L2ineqCyc} and \eqref{dircycBk}).
This inequality is used in Theorem~\ref{THl2polydescCyc} to describe $P = \hat Q_{n+1}$.
For the description of $\hat Q_{\ell}$ with general $\ell$
(to be given in Proposition~\ref{PRdescQell}), 
we need a similar inequality associated with a path.
For a path $L$ in $G_{*}$,
we define its break vertices 
$( i_{1}, j_{1}, \ldots, i_{m+1}, j_{m+1})$
(see below) and 
an inequality associated with $L$ by
\begin{equation}
y_{j} - y_{i} + x(\{ j_{1}, \ldots, j_{m} \}) - x(\{ i_{2}, \ldots, i_{m+1} \}) 
 \leq  \gamma(L),
 \label{L2ineqPathFM}
\end{equation}
where $i$ and $j$ are the initial and terminal vertices of $L$
and
$\gamma(L) = \sum_{e \in L} \gamma(e)$ is the length of $L$
with respect to $\gamma$ in \eqref{edgelenG}.
It is noted that the first and last break vertices, $i_{1}$ and $j_{m+1}$,
do not appear in  \eqref{L2ineqPathFM}.
The inequality \eqref{L2ineqPathFM} generalizes
$y_{j} - y_{i} \leq \gamma_{ij}\sp{(1)}$ 
in \eqref{QP1descPrf2}
and
$ y_{j} - y_{i} + x_{i} - x_{j} \leq \gamma_{ji}\sp{(2)}$
in \eqref{QP2descPrf2}
(see Example~\ref{EXl2graphpath}).

For a (simple or non-simple) path 
$L = e_{1} e_{2} \cdots e_{\ell-1} e_{\ell}$
from $i$ to $j$ $(\neq i)$ in $G_{*}$,
the break vertices 
$( i_{1}, j_{1}, \ldots, i_{m+1}, j_{m+1})$
of $L$ are defined as follows.
The index $m$ is equal to the number of $E_{2}\sp{\circ}$-intervals in $L$;
in particular, $m=0$ if $L$ has no edge from $E_{2}\sp{\circ}$.
We define $i_{1}:= i$ and $j_{m+1} := j$.
Suppose that $L$ has edges from $E_{2}\sp{\circ}$, that is, $m \geq 1$.
Let
$j_{1}$ be the vertex at which 
the first $E_{2}\sp{\circ}$-interval starts;
we have $j_{1}=i_{1}$ 
if the first edge $e_{1}$ belongs to $E_{2}\sp{\circ}$.
For each $r=1,2,\ldots,m$,
the $r$-th $E_{2}\sp{\circ}$-interval 
starts at vertex $j_{r}$ and ends at vertex $i_{r+1}$:
\begin{align} 
L = & \
 |i_{1}| \underbrace{e_{1} \cdots e_{t_{1}}}_{E_{1}} |j_{1}|
    \underbrace{e_{t_{1}+1} \cdots e_{s_{2}-1}}_{E_{2}\sp{\circ}}
 |i_{2}| \underbrace{e_{s_{2}} \cdots e_{t_{2}}}_{E_{1}} |j_{2}| 
  \underbrace{e_{t_{2}+1} \cdots e_{s_{3}-1}}_{E_{2}\sp{\circ}} |i_{3}| e_{s_{3}} \cdots 
\nonumber \\ & \
\cdots e_{s_{m}-1}  
 |i_{m}| \underbrace{e_{s_{m}} \cdots e_{t_{m}}}_{E_{1}}
 |j_{m}| \underbrace{e_{t_{m}+1} \cdots e_{s_{m+1}-1}}_{E_{2}\sp{\circ}} 
 |i_{m+1}| \underbrace{e_{s_{m+1}} \cdots e_{\ell}}_{E_{1}}
 |j_{m+1}| .
 \label{dirpathBk}
\end{align}
If
$e_{1} \in E_{2}\sp{\circ}$,
the first interval 
$e_{1} \cdots e_{t_{1}}$ 
of $E_{1}$-edges
is empty, in which case
we have $t_{1}=0$.
Symmetrically,
if $e_{\ell} \in E_{2}\sp{\circ}$,
the last interval 
$e_{s_{m+1}} \cdots e_{\ell}$ 
of $E_{1}$-edges
is empty, in which case
we have $s_{m+1}=\ell + 1$.
With this convention we have
$0 \leq t_{1}$, \ 
$s_{r} \leq  t_{r}$ $(r=2,3,\ldots,m)$, \ $s_{m+1} \leq \ell +1$
for $E_{1}$-intervals,  and
$t_{r} +1 \leq s_{r+1} -1$
$(r=1,2,\ldots,m)$ 
for $E_{2}\sp{\circ}$-intervals.

The inequality \eqref{L2ineqPathFM}
for a path $L$ is illustrated in the following example,
whereas 
the inequality \eqref{L2ineqCycFM}  for a mixed cycle $C$ 
has been illustrated in Example~\ref{EXl2graphcyc}.

\begin{example} \rm \label{EXl2graphpath}
The inequality \eqref{L2ineqPathFM}
for a path $L$ is illustrated for some cases,
\begin{itemize}
\item 
If $L$ consists of a single edge $(i,j)_{1} \in E_{1}$,
then $m=0$, $(i_{1}, j_{1})=(i,j)$, 
and the inequality \eqref{L2ineqPathFM} reduces to  
$y_{j} - y_{i} \leq \gamma_{ij}\sp{(1)}$ 
in \eqref{QP1descPrf2}.

\item 
If $L$ is a single edge $(i,j)_{2}$ from $E_{2}\sp{\circ}$,
then $m=1$,
$(i_{1}, j_{1})=(i,i)$, 
$(i_{2}, j_{2})=(j,j)$, 
and \eqref{L2ineqPathFM} reduces to  
$ y_{j} - y_{i} + x_{i} - x_{j} \leq \gamma_{ji}\sp{(2)}$
in \eqref{QP2descPrf2}.

\item 
For a simple path
$L = (i,1)_{2} (1,2)_{1} (2,j)_{2}$
with $m=2$,
$(i_{1}, j_{1})=(i,i)$,
$(i_{2}, j_{2})=(1,2)$, and
$(i_{3}, j_{3})=(j,j)$,
the inequality \eqref{L2ineqPathFM} is given by
\begin{equation*} %%\label{FM2term1212}
 y_{j} - y_{i} + x_{i} + x_{2} - x_{1} - x_{j} \leq
 \gamma_{1i}\sp{(2)} + \gamma_{12}\sp{(1)} + \gamma_{j2}\sp{(2)} .
\end{equation*}

\item 
For a simple path
$L = (i,1)_{1} (1,2)_{2} (2,3)_{1} (3,j)_{2}$
with $m=2$,
$(i_{1}, j_{1})=(i,1)$,
$(i_{2}, j_{2})=(2,3)$, and
$(i_{3}, j_{3})=(j,j)$,
the inequality \eqref{L2ineqPathFM} is given by
\begin{equation*} %%\label{FM2term1212}
 y_{j} - y_{i} + x_{1} + x_{3} - x_{2} - x_{j} \leq
 \gamma_{i1}\sp{(1)} + \gamma_{21}\sp{(2)} + 
 \gamma_{23}\sp{(1)} + \gamma_{j3}\sp{(2)} .
\end{equation*}

\item
For a non-simple path 
$L = (i,1)_{1} (1,2)_{2} (2,1)_{1} (1,j)_{2}$
with $m=2$,
$(i_{1}, j_{1})=(i,1)$, 
$(i_{2}, j_{2})=(2,1)$, and
$(i_{3}, j_{3})=(j,j)$,
the inequality \eqref{L2ineqPathFM} is given by
\begin{equation} \label{FMcoef2}
 y_{j} - y_{i} + 2 x_{1} - x_{2} - x_{j} 
\leq 
\gamma_{i1}\sp{(1)} + \gamma_{21}\sp{(2)} + \gamma_{21}\sp{(1)}  + \gamma_{j1}\sp{(2)},
\end{equation}
in which $x_{1}$ appears with coefficient $2$.
\finbox
\end{itemize}
\end{example}

Recall notations
$\hat Q$ in \eqref{l2Qhat}
and its projection 
$\hat Q_{\ell}$
to the space of
$(x,y_{[\ell]}) \in  \RR\sp{2n - \ell +1}$,
where $y_{[\ell]} := (y_{\ell}, y_{\ell+1}, \ldots, y_{n})$.
The following proposition states that
each $\hat Q_{\ell}$ can be described by \eqref{L2ineqCycFM} 
for a suitably chosen family $\mathcal{C}_{\ell}$ of simple mixed cycles in $G_{*}$ and 
\eqref{L2ineqPathFM} 
for a suitably chosen family $\mathcal{P}_{\ell}$ of simple paths in $G_{*}$ 
connecting $i \in N_{\ell}$ to $j \in N_{\ell}$,
where
$N_{\ell}:= \{ \ell, \ell+1, \ldots, n \}$.

\begin{proposition} \label{PRdescQell}
Let $1 \leq \ell \leq n+1$.  We have
\begin{align}
\hat Q_{\ell} = 
 \{ (x,y_{[\ell]}) \mid \ &  
\mbox{\rm \eqref{L2ineqCycFM} 
for every (simple mixed cycle) $C \in \mathcal{C}_{\ell}$,}
\nonumber \\ & 
\mbox{\rm 
\eqref{L2ineqPathFM} for every (simple path) $L \in \mathcal{P}_{\ell}$} 
 \}
 \label{L2QelldescFM}
\end{align}
for a family $\mathcal{C}_{\ell}$
of simple mixed cycles
in $G_{*}$
and a family $\mathcal{P}_{\ell}$
of simple paths 
in $G_{*}$
from a vertex in $N_{\ell}$ to another vertex in $N_{\ell}$.
\finboxARX
\end{proposition}

In Section~\ref{SCproofinduct}
we prove Proposition~\ref{PRdescQell}
by induction on $\ell$ 
on the basis of the Fourier--Motzkin elimination.
The expression \eqref{L2QelldescFM}
for $\ell = 1$ is true 
with the choice of
$\mathcal{C}_{1} =\emptyset$
and
$\mathcal{P}_{1} =E_{1} \cup E_{2}\sp{\circ}$,
because 
$\hat Q_{1} = \hat Q$
is described  by \eqref{QP1descPrf2} and \eqref{QP2descPrf2},
which are special cases of \eqref{L2ineqPathFM} 
as explained in Example~\ref{EXl2graphpath}.
The expression \eqref{L2QelldescFM} for $\ell = n+1$
implies the expression \eqref{L2PdescCyc}
in Theorem~\ref{THl2polydescCyc},
since 
$\hat Q_{n+1} = P$ and 
$\mathcal{P}_{n+1} = \emptyset$
(by $N_{n+1} = \emptyset$).
Before describing the general induction step, 
we illustrate in Section~\ref{SCelimy1} the first step to eliminate $y_{1}$.

\begin{remark} \rm \label{RMbKpathcycle}
The definition of break vertices for a path in \eqref{dirpathBk} 
is consistent with that
for a mixed cycle in \eqref{dircycBk}.
A mixed cycle 
$C = e_{1} e_{2} \cdots e_{\ell-1} e_{\ell}$
with $e_{1} \in E_{1}$ and $e_{\ell} \in E_{2}\sp{\circ}$
may be identified with a path
$L = e_{1} e_{2} \cdots e_{\ell-1} e_{\ell}$
from $i$ to $j$,
where $i$ is the tail of $e_{1}$ and $j$ is a copy of $i$.
If $( i_{1}, j_{1}, \ldots, i_{m+1}, j_{m+1})$
denotes the break vertices of $L$, 
the break vertices of $C$ are given by 
$( i_{1}, j_{1}, \ldots, i_{m}, j_{m})$.
Moreover, the inequality \eqref{L2ineqCycFM} for $C$
coincides (formally) with \eqref{L2ineqPathFM} for $L$
with the understanding of
$y_{j}=y_{i}$ and $i_{m+1} = j = i = i_{1}$.
\finbox
\end{remark}

\subsubsection{Eliminating $y_{1}$}
\label{SCelimy1}

Our proof of Proposition~\ref{PRdescQell}
is based on 
the Fourier--Motzkin elimination applied to
\eqref{QP1descPrf2}--\eqref{QP2descPrf2}.
In this section we describe the first step to eliminate $y_{1}$.

To eliminate $y_{1}$ we classify the inequalities into six groups as
\begin{align} 
& 
 y_{1} - y_{i} \leq \gamma_{i1}\sp{(1)} 
 \quad ((i,1) \in E_{1}),
\label{QP1+descPrf3} 
\\ &
 y_{j} - y_{1} \leq \gamma_{1j}\sp{(1)} 
 \quad ((1,j) \in E_{1}),
\label{QP1-descPrf3} 
\\ &
 y_{j} - y_{i} \leq \gamma_{ij}\sp{(1)} 
 \quad (i \ne 1, j \ne 1, (i,j) \in E_{1}),
\label{QP10descPrf3} 
\\ &
 y_{1} - y_{i} +  x_{i} - x_{1} \leq \gamma_{1i}\sp{(2)} 
\quad ((1,i) \in E_{2}),
\label{QP2+descPrf3} 
\\ &
 y_{j} - y_{1} + x_{1} - x_{j} \leq \gamma_{j1}\sp{(2)} 
\quad ((j,1) \in E_{2}),
\label{QP2-descPrf3} 
\\ &
 y_{j} - y_{i} + x_{i} - x_{j} \leq \gamma_{ji}\sp{(2)} 
\quad (i \ne 1, j \ne 1,(j,i) \in E_{2}).
\label{QP20descPrf3} 
\end{align}
Note that $y_{1}$ appears with coefficient ``$+1$'' 
in \eqref{QP1+descPrf3} and \eqref{QP2+descPrf3},
and with coefficient ``$-1$'' 
in \eqref{QP1-descPrf3} and \eqref{QP2-descPrf3},
while $y_{1}$ does not appear in \eqref{QP10descPrf3} and \eqref{QP20descPrf3}.
Thus there are four types of combinations to eliminate $y_{1}$, namely,
\eqref{QP1+descPrf3} $+$ \eqref{QP1-descPrf3},
\eqref{QP2+descPrf3} $+$ \eqref{QP2-descPrf3},
\eqref{QP1+descPrf3} $+$ \eqref{QP2-descPrf3},
and
\eqref{QP2+descPrf3} $+$ \eqref{QP1-descPrf3}.

\begin{itemize}
\item
\eqref{QP1+descPrf3} $+$ \eqref{QP1-descPrf3}:
This combination gives rise to 
\begin{equation*} %% \label{FMredund1}
y_{j} - y_{i} \leq \gamma_{i1}\sp{(1)} + \gamma_{1j}\sp{(1)} .
\end{equation*}
This inequality coincides with
\eqref{L2ineqPathFM} for 
$L = (i,1)_{1} (1,j)_{1}$,
for which $m=0$ and $(i_{1}, j_{1})=(i,j)$.

\item
\eqref{QP2+descPrf3} $+$ \eqref{QP2-descPrf3}:
This combination gives rise to 
\begin{equation*}  %% \label{FMredund2}
 y_{j} - y_{i} + x_{i} - x_{j} \leq \gamma_{j1}\sp{(2)} + \gamma_{1i}\sp{(2)} .
\end{equation*}
This inequality coincides with
\eqref{L2ineqPathFM} for 
$L = (i,1)_{2} (1,j)_{2}$,
for which $m=1$ and $(i_{1}, j_{1})=(i,i)$ and $(i_{2}, j_{2})=(j,j)$.

\item
\eqref{QP1+descPrf3} $+$ \eqref{QP2-descPrf3}:
The addition of \eqref{QP1+descPrf3} and \eqref{QP2-descPrf3} generates
\begin{equation*} % \label{FM2term12}
 y_{j} - y_{i} + x_{1} - x_{j} \leq 
 \gamma_{i1}\sp{(1)} + \gamma_{j1}\sp{(2)} .
\end{equation*}
If $i \neq j$, this inequality coincides with
\eqref{L2ineqPathFM} for 
$L = (i,1)_{1} (1,j)_{2}$,
for which $m=1$ and $(i_{1}, j_{1})=(i,1)$ 
and $(i_{2}, j_{2})=(j,j)$.
If $i=j$, the above inequality reduces to 
\begin{equation*}  %%\label{FM2term12ii}
 x_{1} - x_{i} \leq 
 \gamma_{i1}\sp{(1)} + \gamma_{i1}\sp{(2)} ,
\end{equation*}
which coincides with
\eqref{L2ineqCycFM} for 
$C = (i,1)_{1} (1,i)_{2}$,
for which $m=1$ and $(i_{1}, j_{1})=(i,1)$.

\item
\eqref{QP2+descPrf3} $+$ \eqref{QP1-descPrf3}:
The addition of \eqref{QP2+descPrf3} and \eqref{QP1-descPrf3} generates
\begin{equation*} %% \label{FM2term21}
 y_{j} - y_{i} + x_{i} - x_{1} \leq 
 \gamma_{1j}\sp{(1)} + \gamma_{1i}\sp{(2)} .
\end{equation*}
If $i \neq j$, this inequality coincides with
\eqref{L2ineqPathFM} for 
$L = (i,1)_{2} (1,j)_{1}$,
for which $m=1$ and $(i_{1}, j_{1})=(i,i)$ and $(i_{2}, j_{2})=(1,j)$.
If $i=j$, the above inequality reduces to 
\begin{equation*} %% \label{FM2term21ii}
  x_{i} - x_{1} \leq 
  \gamma_{1i}\sp{(1)} + \gamma_{1i}\sp{(2)} ,
\end{equation*}
which coincides with
\eqref{L2ineqCycFM} for 
$C = (i,1)_{2} (1,i)_{1}= (1,i)_{1}(i,1)_{2}$,
for which $m=1$ and $(i_{1}, j_{1})=(1,i)$.
\end{itemize}

Thus the system of inequalities for 
$(y_{2}, y_{3}, \ldots, y_{n};  \allowbreak  x_{1}, x_{2}, \ldots, x_{n})$ 
is given by 
\eqref{QP10descPrf3}, \eqref{QP20descPrf3},
and the inequalities derived above.
The interval of $y_{1}$ is given by \eqref{FMp1} as 
\begin{align*} 
&  \max \left\{  
   \max_{(1,j) \in E_{1}} \{ -\gamma_{1j}\sp{(1)} + y_{j}  \} ,
   \max_{(j,1) \in E_{2}} \{ -\gamma_{j1}\sp{(2)} + y_{j} + x_{1} - x_{j}  \} 
 \right\} 
\\ &  \leq y_{1} \leq  
 \min \left\{  
  \min_{(i,1) \in E_{1}} \{ \gamma_{i1}\sp{(1)} + y_{i}  \} ,
  \min_{(1,i) \in E_{2}} \{ \gamma_{1i}\sp{(2)} + y_{i} - x_{i} + x_{1}  \} 
 \right\}, 
\end{align*}
although this expression plays no role in the proof 
of Proposition~\ref{PRdescQell}.

\subsubsection{Proof of Proposition~\ref{PRdescQell}}
\label{SCproofinduct}

We prove \eqref{L2QelldescFM} in Proposition~\ref{PRdescQell}
by induction on $\ell=1,2,\ldots,n+1$.
As already mentioned right after Proposition~\ref{PRdescQell},
the expression \eqref{L2QelldescFM}
for $\ell = 1$ is true 
with the choice of
$\mathcal{C}_{1} =\emptyset$
and
$\mathcal{P}_{1} =E_{1} \cup E_{2}\sp{\circ}$.

As the induction hypothesis, suppose that 
\eqref{L2QelldescFM} is true for $\ell$, where $1 \leq \ell \leq n$.
That is, we assume that we have
$(\mathcal{C}_{\ell}, \mathcal{P}_{\ell})$ 
such that
$\hat Q_{\ell}$ is described by
\eqref{L2ineqCycFM} for $C \in \mathcal{C}_{\ell}$
and 
\eqref{L2ineqPathFM} for $L \in \mathcal{P}_{\ell}$.
Since
$\hat Q_{\ell+1}$ is the projection of $\hat Q_{\ell}$
along the coordinate axis of $y_{\ell}$,
we can obtain an inequality system for 
$\hat Q_{\ell+1}$ by eliminating the variable $y_{\ell}$
from the inequalities in 
\eqref{L2ineqPathFM} for $L \in \mathcal{P}_{\ell}$.
It is noted that 
the inequalities in 
\eqref{L2ineqCycFM} for $C \in \mathcal{C}_{\ell}$,
being free from $y_{\ell}$,
are not involved in the elimination process.

Let $\mathcal{P}_{\ell}\sp{+}$ denote 
the set of paths in $\mathcal{P}_{\ell}$
ending at vertex $\ell$ and, similarly,
let $\mathcal{P}_{\ell}\sp{-}$ be
the set of paths in $\mathcal{P}_{\ell}$
starting at vertex $\ell$.
For $L\sp{+} \in \mathcal{P}_{\ell}\sp{+}$
and
$L\sp{-} \in \mathcal{P}_{\ell}\sp{-}$,
express the corresponding inequalities as
\begin{align}
y_{\ell} - y_{i} + x( J\sp{+} ) - x( I\sp{+} ) 
 & \leq  \gamma(L\sp{+}),
 \label{L2ineqPathFMellp}
\\
y_{j} - y_{\ell} + x( J\sp{-} ) - x( I\sp{-} ) 
 & \leq  \gamma(L\sp{-}),
 \label{L2ineqPathFMellm}
\end{align}
where $i$ $(\neq \ell)$ is the starting vertex of $L\sp{+}$,
$j$ $(\neq \ell)$ is the end vertex of $L\sp{-}$,
and 
$(I\sp{+}, J\sp{+})$ and $(I\sp{-}, J\sp{-})$
are determined from 
the break vertices of $L\sp{+}$ and $L\sp{-}$, respectively.
The operation of the Fourier--Motzkin elimination 
is equivalent to adding \eqref{L2ineqPathFMellp} and \eqref{L2ineqPathFMellm},
resulting in 
\begin{align}
y_{j} - y_{i} +  [x( J\sp{+} )+ x( J\sp{-} )] - [x( I\sp{+} ) + x( I\sp{-} )] 
 & \leq  \gamma(L\sp{+}) + \gamma(L\sp{-}).
 \label{L2ineqPathFMelladded}
\end{align}
Then we obtain the following description of $\hat Q_{\ell+1}$:
\begin{align}
\hat Q_{\ell+1} = 
 \{ (x,y_{[\ell+1]}) \mid \ &  
\mbox{\rm \eqref{L2ineqCycFM} for all $C \in \mathcal{C}_{\ell}$},  
\nonumber \\ & 
\mbox{\rm 
\eqref{L2ineqPathFM} for all $L \in \mathcal{P}_{\ell} 
  \setminus (\mathcal{P}_{\ell}\sp{+} \cup \mathcal{P}_{\ell}\sp{-})$}, 
\nonumber \\ & 
\mbox{\rm 
\eqref{L2ineqPathFMelladded}
 for all $(L\sp{+}, L\sp{-}) \in \mathcal{P}_{\ell}\sp{+} \times \mathcal{P}_{\ell}\sp{-}$}
 \}.
 \label{L2QelldescFMnext}
\end{align}

\begin{remark} \rm \label{RMdecneeded}
Here is a remark to motivate 
our subsequent argument.
The inequality \eqref{L2ineqPathFMelladded}
for a pair 
$(L\sp{+}, L\sp{-}) \in \mathcal{P}_{\ell}\sp{+} \times \mathcal{P}_{\ell}\sp{-}$
corresponds to \eqref{FMp2pn} in the general framework of
the Fourier--Motzkin elimination.
As such, the inequality \eqref{L2ineqPathFMelladded}
is legitimate to describe $\hat Q_{\ell+1}$,
but it may contain coefficients of $\pm 2$.
More specifically, a coefficient of $2$ appears
if $(J\sp{+} \cap J\sp{-})\setminus (I\sp{+} \cup I\sp{-}) \neq \emptyset$.
In contrast, the coefficients 
must be taken from $\{ -1,0,+1 \}$
in \eqref{L2QelldescFM}.
In the following we are going to find a family of
inequalities of admissible forms
that implies \eqref{L2ineqPathFMelladded}.
This family of inequalities are used to update
$(\mathcal{C}_{\ell}, \mathcal{P}_{\ell})$ 
to
$(\mathcal{C}_{\ell+1}, \mathcal{P}_{\ell+1})$. 
\finbox
\end{remark}

Consider the concatenation (series connection) of
$L\sp{+}$ and $L\sp{-}$,
which is a path $\tilde L$ from $i$ to $j$ if $i \neq j$
or a cycle $\tilde C$ if $i=j$.
(The paths $L\sp{+}$ and $L\sp{-}$ 
may possibly have common edges,
and in such a case, it will be more precise to call 
$\tilde L$ a walk, although we refer to it as a path.) 
For the sake of description, 
we assume that we have a path $\tilde L$
with break vertices 
$( i_{1}, j_{1}, \ldots, i_{m+1}, j_{m+1})$,
while the other case of a cycle
can be treated in a similar manner.
This path is not necessarily simple.
In particular, there is a possibility of
$(I\sp{+} \cup J\sp{+}) \cap (I\sp{-} \cup J\sp{-}) \neq \emptyset$.
If $I\sp{+} \cap I\sp{-} \neq \emptyset$ or
$J\sp{+} \cap J\sp{-} \neq \emptyset$,
the associated inequality 
\eqref{L2ineqPathFM} contains coefficients other than $\{ -1,0,+1 \}$
(see \eqref{FMcoef2} for an example).

We can decompose the path $\tilde L$ from $i$ to $j$ into 
a union of a simple path $L_{0}$ from $i$ to $j$ and 
a family of simple cycles, say, 
$\{ C_{q} \mid q \in K \}$,
where $K$ can be empty.
We have
\begin{equation}  \label{gammaLCqLpLm}
\gamma(L_{0}) + \sum_{q \in K} \gamma(C_{q}) 
=  \gamma(L\sp{+}) + \gamma(L\sp{-}).
\end{equation}
For the simple path $L_{0}$ we can consider an inequality in \eqref{L2ineqPathFM},
which we denote as
\begin{equation} \label{L2ineqCycLprimFM}
  y_{j} - y_{i} + x(J_{0}) - x(I_{0}) \leq  \gamma(L_{0}),
\end{equation}
where $I_{0} \cup J_{0}$ is determined from 
the break vertices of $L_{0}$.
Each cycle $C_{q}$ is simple but may or may not be mixed.
If $C_{q}$ is mixed, 
we can consider an inequality of the form of \eqref{L2ineqCycFM}
associated with $C_{q}$, 
which we express as
\begin{equation} \label{L2ineqCycCqFM}
 x(J_{q}) - x(I_{q}) \leq  \gamma(C_{q}),
\end{equation}
where $I_{q} \cup J_{q}$
is the break vertices of $C_{q}$.
Denote the mixed cycles
among $\{ C_{q} \mid q \in K \}$
by $\{ C_{q} \mid q \in K_{*} \}$,
where $K_{*} \subseteq K$.
Then we have an inequality \eqref{L2ineqCycCqFM} for each $q \in K_{*}$.

We observe a simple counting relation using the notation 
$\varepsilon(\cdot, \cdot)$ in \eqref{memberindicator1}.

\begin{lemma} \label{LMcntPath}
For each $h \in N$, we have
\begin{equation}  
[\varepsilon(h,  J\sp{+} )+ \varepsilon(h,  J\sp{-} )] 
  - [\varepsilon(h,  I\sp{+} ) + \varepsilon(h,  I\sp{-} )]
 = \!\!  \sum_{q \in K_{*} \cup \{ 0 \}}  \!\!
 ( \varepsilon(h, J_{q})  - \varepsilon(h, I_{q}) ).
\label{L2ineqPathCycFMcount}
\end{equation}
\end{lemma}

\begin{proof}
Let $h \in N$.
The left-hand side of \eqref{L2ineqPathCycFMcount}
is equal to $+2$ if $h \in J\sp{+} \cap J\sp{-}$, and
it is equal to $0$ if $h \in J\sp{+} \cap I\sp{-}$, etc.:
\[
\begin{array}{c|ccc}
\mbox{LHS of \eqref{L2ineqPathCycFMcount}}
 & h \in I\sp{-} \phantom{a} & h \in J\sp{-} \phantom{a} 
        & h \notin I\sp{-} \cup J\sp{-} 
\\ \hline
h \in I\sp{+} & -2 & \phantom{+}0 & -1
\\
h \in J\sp{+} & \phantom{+}0 & +2 & +1
\\
h \notin I\sp{+} \cup J\sp{+} & -1 & +1 & \phantom{+}0
\\ \hline
\end{array}
\]
Next we consider the right-hand side of \eqref{L2ineqPathCycFMcount}.
Let $e_{\rm in}\sp{+}$ and $e_{\rm out}\sp{+}$ 
be the edges of $L\sp{+}$, if any, 
that enter and leave vertex $h$, respectively.
Define $e_{\rm in}\sp{-}$ and $e_{\rm out}\sp{-}$ similarly for $L\sp{-}$.

\begin{itemize}
\item
If $h \in J\sp{+} \cap J\sp{-}$,
then 
$\{ e_{\rm in}\sp{+}, e_{\rm in}\sp{-} \} \subseteq E_{1}$
and
$\{ e_{\rm out}\sp{+}, e_{\rm out}\sp{-} \} \subseteq E_{2}\sp{\circ}$,
and there exist precisely two $q$'s such that $h \in J_{q}$.
Hence the right-hand side of \eqref{L2ineqPathCycFMcount} is equal to $+2$.

\item
If $h \in J\sp{+} \cap I\sp{-}$, then 
$\{ e_{\rm in}\sp{+}, e_{\rm out}\sp{-} \} \subseteq E_{1}$
and
$\{ e_{\rm in}\sp{-}, e_{\rm out}\sp{+} \} \subseteq E_{2}\sp{\circ}$.
Two cases can be distinguished.
In the first case, 
there are distinct $q'$ and $q''$
such that
$h \in J_{q'}$ and $h \in I_{q''}$.
In the second case,
$h$ is not contained in any of $I_{q} \cup J_{q}$,
which occurs when
$\{ e_{\rm in}\sp{+}, e_{\rm out}\sp{-} \}$
is used by some $C_{q'}$ (or $L_{0}$) and
$\{ e_{\rm in}\sp{-}, e_{\rm out}\sp{+} \}$
is used by another $C_{q''}$ (or $L_{0}$).
In either case, the right-hand side of \eqref{L2ineqPathCycFMcount} is equal to $0$.

\item
If $h \notin I\sp{+} \cup J\sp{+}$ and $h \notin I\sp{-} \cup J\sp{-}$,
then 
$\{ e_{\rm in}\sp{+}, e_{\rm out}\sp{+} \}$
is contained in $E_{1}$ or $E_{2}\sp{\circ}$,
and similarly,
$\{ e_{\rm in}\sp{-}, e_{\rm out}\sp{-} \}$
is contained in 
$E_{1}$ or $E_{2}\sp{\circ}$.
Suppose, for example, that
$\{ e_{\rm in}\sp{+}, e_{\rm out}\sp{+} \} \subseteq E_{1}$
and 
$\{ e_{\rm in}\sp{-}, e_{\rm out}\sp{-} \} \subseteq E_{2}\sp{\circ}$.
Two cases can be distinguished:
There are distinct $q'$ and $q''$
such that
$h \in J_{q'}$ and $h \in I_{q''}$,
or else
$h$ is not contained in any of $I_{q} \cup J_{q}$.
In either case, the right-hand side of \eqref{L2ineqPathCycFMcount} is equal to $0$.
\end{itemize}
Similar arguments for other cases show that 
the right-hand side of \eqref{L2ineqPathCycFMcount}
coincides with the left-hand side of \eqref{L2ineqPathCycFMcount}.
\qedJIAM
\end{proof}

\begin{lemma} \label{LMineqimplyPath}
Inequality \eqref{L2ineqPathFMelladded}
is implied by 
\eqref{L2ineqCycLprimFM} for $L_{0}$ and
\eqref{L2ineqCycCqFM} for $C_{q}$ for all $q \in K_{*}$. 
\end{lemma}
\begin{proof}
Note first that
$\gamma(C_{q}) \geq 0$ if $C_{q}$ is not mixed, 
since neither $G_{1}$ nor $G_{2}\sp{\circ}$
contains negative cycles.
Then it follows from \eqref{gammaLCqLpLm} that
\begin{equation}  \label{gammaLCqmixLpLm}
\gamma(L_{0}) + \sum_{q \in K_{*}} \gamma(C_{q}) 
  \leq  \gamma(L\sp{+}) + \gamma(L\sp{-}).
\end{equation}
By adding \eqref{L2ineqCycLprimFM} and
\eqref{L2ineqCycCqFM} for $q \in K_{*}$ and using
\eqref{gammaLCqmixLpLm}, we obtain
\[
y_{j} - y_{i} + 
\! \sum_{q \in K_{*} \cup \{ 0 \}} \! ( x(J_{q}) - x(I_{q}) ) 
 \leq  \gamma(L_{0}) + \sum_{q \in K_{*}} \gamma(C_{q}) 
  \leq  \gamma(L\sp{+}) + \gamma(L\sp{-}).
\]
This implies \eqref{L2ineqPathFMelladded}, since 
\[
\sum_{q \in K_{*} \cup \{ 0 \} }  ( x(J_{q}) - x(I_{q}) )
 = [x( J\sp{+} )+ x( J\sp{-} )] - [x( I\sp{+} ) + x( I\sp{-} )]
\]
by Lemma~\ref{LMcntPath}.
\qedJIAM
\end{proof}

For each pair 
$(L\sp{+}, L\sp{-}) \in \mathcal{P}_{\ell}\sp{+} \times \mathcal{P}_{\ell}\sp{-}$,
we obtain a family 
$\mathcal{C}\sp{(L\sp{+}, L\sp{-})} := \{ C_{q} \mid q \in K_{*} \}$
of simple mixed cycles,
and also 
$\mathcal{P}\sp{(L\sp{+}, L\sp{-})} := \{ L_{0} \}$,
where $\mathcal{P}\sp{(L\sp{+}, L\sp{-})}$ is
defined to be an empty set if 
the concatenation of 
$L\sp{+}$ and $L\sp{-}$ forms a cycle.
With the use of 
$\mathcal{C}\sp{(L\sp{+}, L\sp{-})}$ and
$\mathcal{P}\sp{(L\sp{+}, L\sp{-})}$,
we can rephrase Lemma~\ref{LMineqimplyPath}
more precisely as follows:
The inequality \eqref{L2ineqPathFMelladded}
generated by the elimination operation for
a pair of inequalities indexed by    
$(L\sp{+}, L\sp{-}) \in \mathcal{P}_{\ell}\sp{+} \times \mathcal{P}_{\ell}\sp{-}$
is implied by 
\eqref{L2ineqCycLprimFM} for $\mathcal{P}\sp{(L\sp{+}, L\sp{-})}$
and \eqref{L2ineqCycCqFM} for $\mathcal{C}\sp{(L\sp{+}, L\sp{-})}$.
On the basis of this observation, 
we modify 
$(\mathcal{C}_{\ell}, \mathcal{P}_{\ell})$ 
to
$(\mathcal{C}_{\ell+1}, \mathcal{P}_{\ell+1})$ 
as
\begin{align}
\mathcal{C}_{\ell+1} & = \mathcal{C}_{\ell} \cup 
\bigg(
\bigcup \{ \mathcal{C}\sp{(L\sp{+}, L\sp{-})} \mid 
(L\sp{+}, L\sp{-}) \in \mathcal{P}_{\ell}\sp{+} \times \mathcal{P}_{\ell}\sp{-} \}
\bigg),
 \label{UpdatCyc}
\\
\mathcal{P}_{\ell+1} & = 
(\mathcal{P}_{\ell} \setminus (\mathcal{P}_{\ell}\sp{+} \cup \mathcal{P}_{\ell}\sp{-}))
\cup 
\bigg(
\bigcup \{ \mathcal{P}\sp{(L\sp{+}, L\sp{-})} \mid 
(L\sp{+}, L\sp{-}) \in \mathcal{P}_{\ell}\sp{+} \times \mathcal{P}_{\ell}\sp{-} \}
\bigg).
 \label{UpdatePath}
\end{align}
Then we obtain
\begin{equation*} %% \label{L2QelldescFMnextsupset}
\hat Q_{\ell+1} \supseteq
 \{ (x,y_{[\ell+1]}) \mid   
\mbox{\rm \eqref{L2ineqCycFM} for all $C \in \mathcal{C}_{\ell+1}$, \ 
\eqref{L2ineqPathFM} for all $L \in \mathcal{P}_{\ell+1}$} 
 \}.
\end{equation*}

Finally, we observe that the reverse inclusion ($\subseteq$) is obviously true.
Indeed, the addition of the inequalities in 
\eqref{QP1descPrf2} and \eqref{QP2descPrf2}
along 
$C \in \mathcal{C}_{\ell+1}$ 
(resp., $L \in \mathcal{P}_{\ell+1}$)
results in 
\eqref{L2ineqCycFM} for $C$
(resp., \eqref{L2ineqPathFM} for $L$).
This means that these inequalities are satisfied 
by every element of $\hat Q_{\ell+1}$, 
implying the reverse inclusion ($\subseteq$).

Thus we have completed
the proof of Proposition~\ref{PRdescQell}
by induction on $\ell$,
which in turn 
establishes Theorem~\ref{THl2polydescCyc}.

\subsection{Box-total dual integrality of \LL-convex polyhedra}
\label{SCboxTDI}

To state our result we need to define the concepts of
(box-)total dual integrality 
introduced by Edmonds and Giles \cite{EG84tdi}.

A linear inequality system $Ax \leq  b$
is said to be {\em totally dual integral} 
({\em TDI}) if the entries of $A$ and $b$ are rational numbers and
the minimum in the linear programming duality equation
\begin{equation*}  
  \max\{c\sp{\top} x \mid  Ax \leq b \} \ 
  = \ \min \{y\sp{\top} b \mid  y\sp{\top} A=c\sp{\top}, \ y\geq 0 \}
\end{equation*}
has an integral optimal solution $y$ for every integral vector $c$ 
such that the minimum is finite.  
A linear inequality system $Ax \leq  b$
is said to be {\em box-totally dual integral}  ({\em box-TDI}) 
if the system $[ Ax \leq  b, d \leq x\leq c ]$ 
is TDI for each choice of rational (finite-valued) vectors $c$ and $d$.  
It is known  \cite[Theorem 5.35]{Sch03} that
the system $A x \leq  b$ is box-TDI
if the matrix $A$ is totally unimodular.
A polyhedron is called a {\em box-TDI polyhedron} if it can be
described by a box-TDI system.  
See \cite{CGR21,Cook83,Cook86,Sch86,Sch03}
for more about box-total dual integrality.

We are now ready to state our result.

\begin{theorem} \label{THl2boxTDI}
An \LL-convex polyhedron is box-TDI.
More generally, an \LLnat-convex polyhedron is box-TDI.
\end{theorem}
\begin{proof}
We first consider an \LL-convex polyhedron $P$.
Recall from Section~\ref{SCproofL2ineq}
that an \LL-convex polyhedron $P$ is obtained from 
the polyhedron $Q$ of \eqref{l2Qdef},
which is described by the system
consisting of 
\eqref{QP1P2descPrf},
\eqref{QP1descPrf}, and 
\eqref{QP2descPrf}.
This system can be written as 
\begin{equation}  \label{QP1P2matrix}
  \left[ \begin{array}{ccc}
  -I & I &  I  \\
  I & -I & -I  \\
   O & B_1 & O  \\
   O  & O  & B_2 \\
\end{array} \right]
  \left[ \begin{array}{ccc} x \\ y \\ z  \end{array} \right]  
\leq 
  \left[ \begin{array}{ccc} 0 \\ 0 \\ \gamma\sp{(1)} \\ \gamma\sp{(2)}
 \end{array} \right] , 
 \end{equation}
where each $I$ is the identity matrix of order $n$
and, for $k=1,2$, 
$B_{k}$ is a matrix whose rows are 
$\unitvec{j} - \unitvec{i}$ for $(i,j) \in E_{k}$.
Each $B_{k}$ is totally unimodular.%
\footnote{%%%%%%%%%%%%%%
Matrix $B_{k}$ is the transpose of the incidence matrix 
of the graph $G_{k}$  in Section~\ref{SCgraphrep}.
} %%%%% footnote %%%%%%%
Therefore, the matrix 
$A=$
{\footnotesize 
$\left[ \begin{array}{ccc}
  -I & I &  I  \\
  I & -I & -I  \\
   O & B_1 & O  \\
   O  & O  & B_2 \\
\end{array} \right]$
}%%font
is also totally unimodular,
which implies that the system 
\eqref{QP1P2matrix} is box-TDI,
and hence the polyhedron $Q$
described by \eqref{QP1P2matrix} is box-TDI.
Since the projection of a box-TDI polyhedron
on a coordinate hyperplane is box-TDI
(\cite[Theorem 3.4]{Cook86}, \cite[pp.~323--324]{Sch86}),
the polyhedron $P$ is also box-TDI.
Thus every \LL-convex polyhedron is box-TDI.
By Proposition~\ref{PRfromL2toL2nat}(2),
this implies further that 
every \LLnat-convex polyhedron is box-TDI; see \cite[p.~323]{Sch86}.
\qedJIAM
\end{proof}

Theorem~\ref{THl2boxTDI} enables us to 
apply the results of 
Frank and Murota \cite{FM21boxTDI}
for separable convex minimization on a box-TDI set.
By so doing 
we can obtain min-max formulas for separable convex minimization on an \LLnat-convex set.

\begin{remark} \rm \label{RMboxTDIredund}
Here is a supplementary remark about Theorem~\ref{THl2boxTDI}.
It is known (\cite[p.~323]{Sch86}, \cite[Theorem 2.5]{Cook83})
that box-total dual integrality is maintained 
under Fourier--Motzkin elimination
if the coefficients belong to $\{ -1,0,+1 \}$.
This is the case with our system 
\eqref{QP1descPrf2}--\eqref{QP2descPrf2}.
However, we have discarded redundant inequalities
in the course of the elimination process,
whereas redundant inequalities are often necessary for
a system of inequalities to be box-TDI.
In view of this, 
box-total dual integrality of the system \eqref{L2PdescCyc} 
does not seem to follow from our argument,
although it is likely that the system 
\eqref{L2PdescCyc} 
is, in fact, box-TDI.   
\finbox
\end{remark}

We note in passing 
that the intersection of an \LLnat-convex polyhedron with a box
is not necessarily an \LLnat-convex polyhedron,
although it remains to be an integral polyhedron by Theorem~\ref{THl2boxTDI}.
An example is given below.

\begin{example} \rm \label{EXinterL2natB}
Recall the \LLnat-convex set 
$ S = \{ (0,0,0), (1,1,0), (0,1,1), (1,2,1) \}$
in Example~\ref{EXLnat2poly},
and let $B = \{ 0, 1 \}\sp{3}$ be the unit box of integers.
Then 
$S \cap B = \{ (0,0,0), (1,1,0), (0,1,1) \}$ is not an \LLnat-convex set,
since $S \cap B$ itself is not \Lnat-convex,
and the decomposition  
$S \cap B =S'_{1} + S'_{2}$
with $S'_{1}, S'_{2} \subseteq \ZZ\sp{3}$ 
is possible only if 
$S'_{i} = S \cap B$ and $S'_{j} = \{ (0,0,0) \}$
for $i \neq j$.
A similar statement applies to the polyhedral version.
Indeed, the convex hull $\overline{S}$ of $S$ is an \LLnat-convex polyhedron,
and 
$\overline{S} \cap \overline{B}$
$(=\overline{S \cap B})$ is not an \LLnat-convex polyhedron. 
\finbox
\end{example}

\section{Implications}
\label{SCimpli}

In this section we show 
alternative proofs to some fundamental facts on discrete convexity
by using the results of this paper on 
polyhedral descriptions of \LL- and \LLnat-convexity.
In general terms, we can distinguish between
``inner'' descriptions and ``outer'' descriptions
of combinatorial objects.
For example, a matroid can be defined by an exchange axiom
as well as a submodular (rank) function,
where the former is an inner description and 
latter an outer description.
Such a dual view often affords a deeper understanding,
which is the general recognition behind this section.
The alternative proofs given here are of outer-type, 
while the existing proofs are of inner-type.
We use Theorems \ref{THl2polydesc} and \ref{THlnat2polydesc},
and not their refined versions 
in Theorems \ref{THl2polydescCyc} and \ref{THlnat2polydescCyc}.

\subsection{\LL-convexity and \Lnat-convexity}

It is pointed out recently in \cite{MM21inclinter} 
that a polyhedron $P$ (or a set $S$ of integer vectors)
is L-convex if and only if 
it is both \LL-convex and \Lnat-convex.
Here we show an alternative polyhedral proof,
based on Theorem~\ref{THl2polydesc},
when the polyhedron is full-dimensional.
It is mentioned, however, that Proposition~\ref{PRl2lnatB} itself is an easy fact, which 
can be proved in any way, but the given proof will serve as a prototype
of the polyhedral argument.

\begin{proposition}[\cite{MM21inclinter}] \label{PRl2lnatB}
\quad

\noindent
{\rm (1)}
A polyhedron $P$ $(\subseteq \RR\sp{n})$ is L-convex
if and only if
it is both \LL-convex and \Lnat-convex.

\noindent
{\rm (2)}
A set $S$ $(\subseteq \ZZ\sp{n})$ is L-convex
if and only if
it is both \LL-convex and \Lnat-convex.
\end{proposition}

\begin{proof}(Full-dimensional case) \ 
First note that (2) follows from (1) applied to the convex hull of $S$,
and that the only-if-part of (1) is obvious.
In the following, we prove the if-part of (1), that is, 
if a polyhedron $P$ is both
\LL-convex and \Lnat-convex,
then $P$ is L-convex.
Let $F$ be any facet of $P$ and $\nu$ a normal vector of $F$.
Since $P$ is \LL-convex, we have $\nu = c( \unitvec{I} - \unitvec{J}) $
for some $c \ne 0$ and disjoint $I, J \subseteq N$ with $|I|=|J|$
by Theorem~\ref{THl2polydesc} 
(see also Remark~\ref{RMedgeM2}
in Section~\ref{SCpolydescThmS}).
On the other hand,
since $P$ is \Lnat-convex, we have 
$\nu = c (\unitvec{i} - \unitvec{j})$
or 
$\nu = c \unitvec{i}$
for some $c \ne 0$ and $i \ne j$ 
by Theorem~\ref{THlnatpolydesc}.
Then it follows that 
$\nu = c (\unitvec{i} - \unitvec{j})$
for some $c \ne 0$ and $i \ne j$.
Therefore, $P$ is L-convex by 
Theorem~\ref{THlpolydesc}.
\qedJIAM
\end{proof}

\subsection{\LLnat-convexity and multimodularity}
\label{SClnat2mm}

It is proved recently in \cite{MM21inclinter} 
that a polyhedron $P$ (or a set $S$ of integer vectors)
is a box if and only if 
it is both \LLnat-convex and multimodular
(see Appendix \ref{SCmmset}
for the definition of multimodular sets).
Here we show an alternative polyhedral proof,
based on 
Theorem~\ref{THlnat2polydesc},
when the polyhedron is full-dimensional.

\begin{proposition}[{\cite{MM21inclinter}}] \label{PRlnat2mmB}
\quad

\noindent
{\rm (1)}
A polyhedron $P$ $(\subseteq \RR\sp{n})$ is a box of reals
if and only if
it is both \LLnat-convex and multimodular.

\noindent
{\rm (2)}
A set $S$ $(\subseteq \ZZ\sp{n})$ is a box of integers
if and only if
it is both \LLnat-convex and multimodular.
\end{proposition}

\begin{proof}
(Full-dimensional case) \
First note that (2) follows from (1) applied to the convex hull of $S$,
and that the only-if-part of (1) is obvious
(cf., \cite[Proposition~2]{MM19multm}).
We prove the if-part of (1) when $P$ is full-dimensional.
Let $F$ be any facet
of $P$ and $\nu$ a normal vector of $F$.
Since $P$ is \LLnat-convex, we have $\nu = c( \unitvec{I} - \unitvec{J}) $
for some $c \ne 0$ and 
disjoint $I, J \subseteq N$ with $|I|-|J| \in \{ -1,0,1 \}$
by Theorem~\ref{THlnat2polydesc}.
On the other hand,
since $P$ is multimodular, we have $\nu = c \unitvec{I}$
for some $c \ne 0$ and consecutive index set $I \subseteq N$ by 
Theorem~\ref{THmmpolydesc}.
Then it follows that $\nu =  c \unitvec{i}$ for some $c \ne 0$ and $i \in N$.
Therefore, $P$ is a box.
\qedJIAM
\end{proof}

\subsection{\LLnat-convexity and \MMnat-convexity}
\label{SClnat2mnat2}

A nonempty set $S$ $\subseteq \ZZ\sp{n}$ is called 
{\em \MMnat-convex}  
if it can be represented as the intersection of two \Mnat-convex sets.
Similarly, a polyhedron $P$ $\subseteq \RR\sp{n}$
is called  \MMnat-convex
if it is the intersection of two \Mnat-convex polyhedra.
See Remark~\ref{RMmnat2poly} below
for the definitions of \Mnat-convex sets and polyhedra.

It is known \cite[Lemma 5.7]{MS01rel} 
that a set $S$ $(\subseteq \ZZ\sp{n})$  is a box of integers
if and only if 
it is both \LLnat-convex and \MMnat-convex,
and an analogous statement is true for the polyhedral case.
Here we show an alternative polyhedral proof,
based on Theorem~\ref{THlnat2polydesc},
when the polyhedron is full-dimensional.

\begin{proposition}[{\cite[Lemma 5.7]{MS01rel}}]  \label{PRlnat2mnat2B}
\quad

\noindent
{\rm (1)}
A polyhedron $P$ $(\subseteq \RR\sp{n})$ is a box of reals
if and only if
it is both \LLnat-convex and \MMnat-convex.

\noindent
{\rm (2)}
A set $S$ $(\subseteq \ZZ\sp{n})$ is a box of integers
if and only if
it is both \LLnat-convex and \MMnat-convex.
\end{proposition}

\begin{proof}(Full-dimensional case) \ 
First note that (2) follows from (1) applied to the convex hull of $S$,
and that the only-if-part of (1) is obvious.
We prove the if-part of (1) when $P$ is full-dimensional.
Let $F$ be any facet
of $P$ and $\nu$ a normal vector of $F$.
Since $P$ is \LLnat-convex, we have $\nu = c( \unitvec{I} - \unitvec{J}) $
for some $c \ne 0$ and disjoint $I, J \subseteq N$ with $|I|-|J| \in \{ -1,0,1 \}$
by Theorem~\ref{THlnat2polydesc}.
On the other hand,
since $P$ is \MMnat-convex, we have 
$\nu = c \unitvec{I}$
for some $c \ne 0$ and $I \subseteq N$
(see Remark~\ref{RMmnat2poly} below).
Then it follows that $\nu =  c \unitvec{i}$ for some $c \ne 0$ and $i \in N$.
Therefore, $P$ is a box.
\qedJIAM
\end{proof}

\begin{remark} \rm \label{RMmnat2poly}
An \Mnat-convex polyhedron is a synonym of 
a generalized polymatroid
 (g-polymatroid),
which is described as
$\{ x \mid \mu(I)\leq x(I) \leq \rho(I) \ (I \subseteq N) \}$
with a strong pair of supermodular $\mu$ and submodular $\rho$
(see \cite[Section  3.5(a)]{Fuj05book}, \cite[Section 4.8]{Mdcasiam}).
%% 2021-10-24 / 2022-02-13  (Sec.4.8 is correct) %%%
If $\mu$ and $\rho$ are
integer-valued, the corresponding 
g-polymatroid is an integral polyhedron,
and the set of its integer points is called
an \Mnat-convex set.
The intersection of two \Mnat-convex sets (resp., polyhedra) is 
called an \MMnat-convex set (resp., polyhedron).
Hence an \MMnat-convex set  (resp., polyhedron)
is described by a system of inequalities of the form 
$\max\{ \mu_{1}(I), \mu_{2}(I) \} \leq x(I) \leq \min\{  \rho_{1}(I), \rho_{2}(I) \}$,
from which follows that
a normal vector $\nu$ is of the form of
$\nu = c \unitvec{I}$ for some $c \ne 0$ and $I \subseteq N$.
The reader is referred to 
\cite[Section 4.7]{Mdcasiam} 
for more about \Mnat-convex sets.
\finbox
\end{remark}

\begin{remark} \rm \label{RMlnat2mnat2B}
To be precise, 
\cite[Lemma 5.7]{MS01rel} 
does not deal with the polyhedral case stated in (1) 
of Proposition \ref{PRlnat2mnat2B}.
However, the proof there
can be extended almost literally to the polyhedral case.
For completeness, we show the proof 
adapted to the polyhedral case.
That is, we prove here that, if a polyhedron $P$ is both
\LLnat-convex and \MMnat-convex,
then $P$ is a box of reals.
Let $P_1, P_2 \subseteq \RR\sp{n}$ be \Lnat-convex polyhedra such that
$P = P_1 + P_2$,
and let $Q_1, Q_2 \subseteq \RR\sp{n}$ be \Mnat-convex polyhedra 
(g-polymatroids) such that
$P = Q_1 \cap Q_2$.

First we explain the idea of the proof when $P$ is bounded.
Each $P_{k}$ has the unique minimum element $a^k \in P_{k}$ 
and the unique maximum element $b^k \in P_{k}$.
Then $a = a^1 + a^2$ is the unique minimum of $P$
and
$b = b^1 + b^2$ is the unique maximum of $P$,
for which we have
$P \subseteq [a, b]_{\RR}$.
Since
$a, b \in P = Q_1 \cap Q_2$, we have
$a, b \in Q_{k}$ for $k=1,2$,
where $a \leq b$.
This implies
$[a, b]_{\RR} \subseteq Q_{k}$,
as is easily seen from the polyhedral description of an \Mnat-convex polyhedron.
Therefore, $[a, b]_{\RR} \subseteq Q_1 \cap Q_2 = P$.
Thus we have proved $[a, b]_{\RR} = P$.

The general case where $P$ may be unbounded can be treated as follows.
For each $i \in N$, put 
$a_{i} := \inf_{y \in P}y_{i}$ and
$b_{i} := \sup_{y \in P}y_{i}$,
where we have the possibility of 
$a_{i}=-\infty$ and/or 
$b_{i}=+\infty$.
 Obviously, $P \subseteq [a, b]_{\RR}$ holds.
 To prove $[a, b]_{\RR} \subseteq P$, take any $x \in [a, b]_{\RR}$. 
 For each $i \in N$, there exist vectors $p\sp{i}, q\sp{i} \in P$
such that $p\sp{i}_{i} \leq x_{i} \leq q\sp{i}_{i}$,
where $p\sp{i}_{i}$, $x_{i}$, and $q\sp{i}_{i}$ denote the $i$th component of 
vectors $p\sp{i}$, $x$, and $q\sp{i}$, respectively.
Since $p\sp{i}, q\sp{i} \in P = P_1 + P_2$,
we can express them as
$p\sp{i} = p\sp{i1} + p\sp{i2}$, 
$q\sp{i} = q\sp{i1} + q\sp{i2}$
with some $p\sp{ik}, q\sp{ik} \in P_{k}$ $(k = 1, 2)$.
Consider
\[
p\sp{k} := \bigwedge_{i \in N}p\sp{ik} \in P_{k} , \quad 
q\sp{k} := \bigvee_{i \in N}q\sp{ik} \in P_{k} \quad (k = 1, 2),
\]
and let $p := p\sp{1} + p\sp{2} \in P$ and
$q := q\sp{1} + q\sp{2} \in P$.
Then, for each $i \in N$,
we have 
\[
 p_{i} = p\sp{1}_{i} + p\sp{2}_{i} \leq p\sp{i1}_{i} + p\sp{i2}_{i} = p\sp{i}_{i} \leq x_{i} ,
\quad
 q_{i} = q\sp{1}_{i} + q\sp{2}_{i} \geq q\sp{i1}_{i} + q\sp{i2}_{i} = q\sp{i}_{i} \geq x_{i},
\]
showing $x \in [p, q]_{\RR}$.
Since
$p, q \in P= Q_1 \cap Q_2$, we have
$p, q \in Q_{k}$ for $k=1,2$,
where $p \leq q$.
This implies
$[p, q]_{\RR} \subseteq Q_{k}$,
which follows from the polyhedral description of an \Mnat-convex polyhedron.
Therefore, $x \in [p, q]_{\RR} \subseteq Q_1 \cap Q_2 = P$,
where $x$ is an arbitrarily chosen element of $[a, b]_{\RR}$.
Hence
$[a, b]_{\RR} \subseteq P$.
Thus we complete the proof of $[a, b]_{\RR} = P$.
\finbox
\end{remark}

%\newpage

\section{Conclusion}
\label{SCconcl}

We conclude this paper by summarizing 
our present knowledge about the polyhedral description 
of discrete convex sets in Table \ref{TBpolydesc}.
The polyhedral description of \MM-convex (resp.,\MMnat-convex) sets
is obtained immediately from that of M-convex (resp., \Mnat-convex) sets;
see Remark~\ref{RMmnat2poly}.
The polyhedral description of multimodular sets,
described in Theorem~\ref{THmmpolydesc},
has recently been obtained in \cite{MM21inclinter}.
Polyhedral descriptions are not known for
integrally convex sets
\cite[Section 3.4]{Mdcasiam},
discrete midpoint convex sets \cite{MMTT20dmc},
and directed discrete midpoint convex sets \cite{TT21ddmc}.

%%%%%%%%%% table %%%%%%%%%%
\begin{table}
\begin{center}
\caption{Polyhedral descriptions of discrete convex sets}
\label{TBpolydesc}
\renewcommand{\arraystretch}{1.1}%
\begin{tabular}{l|lc}
  & Vector $a$ for $\langle a,x \rangle \leq b$  &   Ref. 
\\ \hline
 Box (interval) & $\pm \unitvec{i}$    & obvious 
\\ 
L-convex & $\unitvec{j} - \unitvec{i}$   & \cite[Sec.5.3]{Mdcasiam}
\\ 
\Lnat-convex  & $\unitvec{j} - \unitvec{i}$, \  \ $\pm \unitvec{i}$  & \cite[Sec.5.5]{Mdcasiam}   
\\ 
\LL-convex & $\unitvec{J} - \unitvec{I}$  \  ($|I|=|J|$)  & this paper 
\\ 
\LLnat-convex & $\unitvec{J} - \unitvec{I}$  \ ($|I|-|J| \in \{ -1,0,1 \}$)   & this paper 
\\ 
M-convex & $\unitvec{I}$, \ \ $- \unitvec{N} (= -\vecone)$  & \cite[Sec.4.4]{Mdcasiam}  
\\
\Mnat-convex  & $\pm \unitvec{I}$    &  \cite[Sec.4.7]{Mdcasiam}
\\ 
\MM-convex & $\unitvec{I}$, \ \ $- \unitvec{N} (= -\vecone)$   & by M-convex
\\
\MMnat-convex  & $\pm \unitvec{I}$     &  by \Mnat-convex
\\ 
Multimodular  & $\pm \unitvec{I}$ \  ($I$: consecutive)   & \cite{MM21inclinter} 
\\ \hline
\end{tabular}
\renewcommand{\arraystretch}{1.0}%
\end{center}
\end{table}
%%%%%%%%%% table %%%%%%%%%%

\medskip

\noindent {\bf Acknowledgement}. 
This work was partially supported by JSPS KAKENHI Grant Numbers 
JP17K00037, JP20K11697, JP21K04533.
The authors thank Akihisa Tamura for careful reading of the manuscript
and to Satoru Iwata for communicating his personal memorandum \cite{Iwa21adj}.
Comments from the referees were helpful to improve the paper.

%% (Moriguchi=JP17K00037,JP21K04533; Murota=JP20K11697) 

\appendix

\section{Definitions from discrete convex analysis}
\label{SCappend}

\subsection{L-convex and M-convex functions}
\label{SClmfn}

A function 
$f: \ZZ\sp{n} \to \RR \cup \{ +\infty \}$
with $\dom f \not= \emptyset$ 
is called {\em L-convex}
if it is submodular:
\begin{equation*} %% \label{submfn}
f(x) + f(y) \geq f(x \vee y) + f(x \wedge y)
\end{equation*}
for all $x, y \in \ZZ\sp{n}$
 and there exists
$r \in \RR$ such that 
\begin{equation*}  %%\label{shiftlfnZ}
f(x + \mu \vecone) = f(x) + \mu  r
\end{equation*}
for all $x  \in \ZZ\sp{n}$ and $\mu \in \ZZ$.

A function 
$f: \ZZ\sp{n} \to \RR \cup \{ +\infty \}$
with $\dom f \not= \emptyset$ 
is called
{\em \LL-convex}
if it can be represented as the (integral) infimal convolution 
$f_{1} \Box f_{2}$  of 
two L-convex functions
$f_{1}$ and $f_{2}$, that is,
if
\[
f(x) =  (f_{1} \Box f_{2})(x) = \inf\{ f_{1}(y)+ f_{2}(z) 
        \mid x = y+z; \  y,z \in \ZZ\sp{n}  \}
\qquad (x \in \ZZ\sp{n}).
\]
It is known 
\cite[Note 8.37]{Mdcasiam}
that the infimum
is always attained as long as it is finite.

A function 
$f: \ZZ\sp{n} \to \RR \cup \{ +\infty \}$
with $\dom f \not= \emptyset$ 
is called
{\em M-convex}
if it satisfies
the exchange property:
\begin{description}
\item[\Mvexb]
 For any $x, y \in \dom f$  and $i \in \suppp(x-y)$, 
there exists
$j \in \suppm(x-y)$ such that
\begin{equation*}  
f(x) + f(y)   \geq 
 f(x-\unitvec{i}+\unitvec{j}) + f(y+\unitvec{i}-\unitvec{j}) .
\end{equation*}
\end{description}
A nonempty set $S$ is called M-convex
if its indicator function $\delta_{S}$ is an M-convex function.

A function 
$f: \ZZ\sp{n} \to \RR \cup \{ +\infty \}$
with $\dom f \not= \emptyset$ 
is called 
{\em \MM-convex}
if it can be represented as the sum of
two M-convex functions
$f_{1}$ and $f_{2}$, that is,
if
$f(x) =  f_{1}(x) + f_{2}(x)$
$(x \in \ZZ\sp{n})$.
A nonempty set $S$ is called \MM-convex
if it is the intersection of two M-convex sets,
or equivalently, 
if its indicator function $\delta_{S}$ is an \MM-convex function.

The reader is referred to 
\cite{Mdcasiam}  
for characterizations and properties of 
L-, \LL-, M-, and \MM-convex functions.

\subsection{Multimodular sets}
\label{SCmmset}

Let $\calF \subseteq \ZZ\sp{n}$ be the set of vectors defined by
$\calF = \{ -\unitvec{1}, \unitvec{1}-\unitvec{2}, 
  \unitvec{2}-\unitvec{3}, \ldots, 
  \unitvec{n-1}-\unitvec{n}, \unitvec{n} \} $, 
where $\unitvec{i}$ denotes the $i$th unit vector for $i =1,2,\ldots, n$.
A set $S  \subseteq \ZZ\sp{n}$
is said to be {\em multimodular}
if 
\begin{equation*}  %%\label{multimodularsetdef1}
 z+d, \  z+d' \in S \ \Longrightarrow \  z, \ z+d+d' \in S
\end{equation*}
for all $z \in \ZZ\sp{n}$ and all distinct $d, d' \in \calF$.
(The concept of multimodularity is introduced by Hajek \cite{Haj85}
for functions
and its set version is formulated in \cite{MM19multm}.)

It is known \cite{MM19multm,Mmult05}
that multimodular sets are precisely those sets 
which are obtained from 
\Lnat-convex sets by a simple coordinate change.
Define a  bidiagonal matrix 
$D=(d_{ij} \mid 1 \leq i,j \leq n)$ by
$ d_{ii}=1$ $(i=1,2,\ldots,n)$ and
$ d_{i+1,i}=-1$ $(i=1,2,\ldots,n-1)$.
This matrix $D$ is unimodular, and its inverse
$D\sp{-1}$ is an integer 
lower-triangular
matrix with $(D\sp{-1})_{ij}=1$ for $i \geq j$ and 
$(D\sp{-1})_{ij}=0$ for $i < j$.

\begin{proposition}[\cite{MM19multm,Mmult05}]  \label{PRmultmsetLnatset}
A set $S  \subseteq \ZZ\sp{n}$ is multimodular 
if and only if 
it can be represented as  $S = \{ D y \mid y \in T \}$
for some \Lnat-convex set $T$,
where $T$ is uniquely determined from $S$ as 
$T = \{ D\sp{-1} x \mid x \in S \}$.
\finboxARX
\end{proposition}

In accordance with this relation, 
we call a polyhedron $P$
a {\em multimodular polyhedron}
if it can be represented as $P = \{ D y \mid y \in Q \}$
for some \Lnat-convex polyhedron $Q$.
Such $Q$ is uniquely determined from $P$ as $Q = \{ D\sp{-1} x \mid x \in P \}$.

Multimodular sets and polyhedra can be described by inequalities as follows.
A subset $I$ of the index set $N = \{ 1,2,\ldots, n \}$
is said to be {\em consecutive}
if it consists of consecutive numbers, that is,
it is a set of the form 
$I = \{ k, k+1, \ldots, \ell -1, \ell \}$
for some $k \leq \ell$.

\begin{theorem}[\cite{MM21inclinter}] \label{THmmpolydesc}
Let $N = \{ 1,2,\ldots, n  \}$.

\noindent
{\rm (1)}
A nonempty set $P \subseteq \mathbb{R}\sp{n}$ is a multimodular polyhedron
if and only if it can be represented as
$  P =   \{ x \in \RR\sp{n} \mid 
 a_{I} \leq x(I) \leq b_{I}  \ \ (\mbox{\rm $I$: consecutive subset of $N$})   \}$ 
for some $a_{I} \in \RR \cup \{ -\infty \}$ and $b_{I} \in \RR \cup \{ +\infty \}$
indexed by consecutive subsets $I$ of $N$.

\noindent
{\rm (2)}
A nonempty set $S \subseteq \mathbb{Z}\sp{n}$ is a multimodular set
if and only if it can be represented as
$ S =   \{ x \in \ZZ\sp{n} \mid 
 a_{I} \leq x(I) \leq b_{I}  \ \ (\mbox{\rm $I$: consecutive subset of $N$})   \}$
for some $a_{I} \in \ZZ \cup \{ -\infty \}$ and $b_{I} \in \ZZ \cup \{ +\infty \}$
indexed by consecutive subsets $I$ of $N$.
\finboxARX
\end{theorem}

The above theorem implies immediately that a box is multimodular,
which was pointed out first in \cite[Proposition~2]{MM19multm}.

\newpage
\tableofcontents


\begin{thebibliography}{99}

\bibitem{Che17}
Chen, X.:
L$^{\natural}$-convexity and its applications in operations.
Frontiers of Engineering Management {\bf 4}, 283--294 (2017)

\bibitem{CL21dca}
Chen, X.,   Li, M.:
Discrete convex analysis and its applications in operations: A survey.
Management Science {\bf 30}, 1904--1926 (2021) 
%%No.6


\bibitem{CGR21}
Chervet P.,  Grappe, R.,   Robert, L.-H.:
Box-total dual integrality, box-integrality, and equimodular matrices. 
Mathematical Programming, Ser.~A {\bf 188}, 319--349 (2021)

%%published online:  20 May 2020.
%%https://doi.org/10.1007/s10107-020-01514-0




\bibitem{Cook83} 
Cook, W.:
Operations that preserve total dual integrality.
Operations Research Letters {\bf 2},  31--35 (1983)

%% Not "W.J.Cook" in the paper


\bibitem{Cook86} 
Cook, W.:
On box totally dual integral polyhedra.
Mathematical Programming {\bf 34},  48--61 (1986)

%% Not "W.J.Cook" in the paper


\bibitem{EG84tdi} 
Edmonds, J., Giles, R.: 
Total dual integrality of linear inequality systems. 
In: Pulleyblank, W. R. (ed.)   
Progress in Combinatorial Optimization,
pp.~117--129.
Academic Press, New York (1984) 







\bibitem{FM21boxTDI} 
Frank, A., Murota, K.:
A discrete convex min-max formula for box-TDI polyhedra.
Mathematics of Operations Research,
published on-line (October 18, 2021)
https://doi.org/10.1287/moor.2021.1160



\bibitem{FT88} 
Frank, A., Tardos, \'{E}.:
Generalized polymatroids and submodular flows.
Mathematical Programming {\bf 42}, 489--563 (1988)



\bibitem{Fuj05book}
Fujishige, S.:
Submodular Functions and Optimization,
2nd edn.
Annals of Discrete Mathematics {\bf 58},
Elsevier, Amsterdam  (2005)


\bibitem{Gru03} 
Gr{\"u}nbaum, B.:
Convex Polytopes, 2nd edn.
Springer, New York (2003)

%%\memo{Section 15.1: Vector Addition}



\bibitem{Haj85} 
Hajek, B.:
Extremal splittings of point processes.
Mathematics of Operations Research {\bf 10}, 543--556 (1985)



\bibitem{Iwa02adj} 
Iwata, S.:
On matroid intersection adjacency.
Discrete Mathematics, {\bf 242}, 277--281 (2002)



\bibitem{Iwa21adj} 
Iwata, S.: 
On the number of adjacent common bases in matroid intersection.
Personal memorandum (October, 2021)



\bibitem{MM19multm} 
Moriguchi, S., Murota, K.:
On fundamental operations for multimodular functions.
Journal of the Operations Research Society of Japan {\bf  62}, 53--63 (2019)


\bibitem{MM21inclinter} 
Moriguchi, S., Murota, K.:
Inclusion and intersection relations 
between fundamental classes of discrete convex functions.
Submitted for publication;
arXiv: https://arxiv.org/abs/2111.07240
 (November 2021)




\bibitem{MMTT20dmc}
Moriguchi, S., Murota, K.,  Tamura, A.,    Tardella, F.:
Discrete midpoint convexity.
Mathematics of Operations Research {\bf 45}, 99--128 (2020)



\bibitem{Mdca98} 
Murota, K.:
Discrete convex analysis. 
Mathematical Programming  {\bf 83}, 313--371 (1998)



\bibitem{Mdcasiam} 
Murota, K.:
Discrete Convex Analysis.
Society for Industrial and Applied Mathematics, Philadelphia (2003)


\bibitem{Mmult05} 
Murota, K.:
Note on multimodularity and L-convexity.
Mathematics of Operations Research 
{\bf 30}, 658--661 (2005)

\bibitem{Mbonn09} 
Murota, K.:
Recent developments in discrete convex analysis.
In: Cook, W., Lov{\'a}sz, L., Vygen, J. (eds.)
Research Trends in Combinatorial Optimization,
Chapter 11, pp.~219--260. Springer, Berlin (2009) 


\bibitem{Mdcaeco16} 
Murota, K.:
Discrete convex analysis: A tool for economics and game theory.
{Journal of Mechanism and Institution Design} 
{\bf 1}, 151--273 (2016)



\bibitem{Msurvop19} 
Murota, K.:
A survey of fundamental operations on discrete convex functions of various kinds.
Optimization Methods and Software {\bf 36}, 472--518 (2021)





\bibitem{MS01rel} 
Murota, K., Shioura, A.:
Relationship of M-/L-convex functions with 
discrete convex functions by Miller and by Favati--Tardella.
Discrete Applied Mathematics {\bf 115}, 151--176 (2001)


\bibitem{MS14dijk} 
Murota, K., Shioura, A.:
Dijkstra's algorithm and L-concave function maximization.
Mathematical Programming, Series A 
{\bf 145}, 163--177 (2014)




\bibitem{MT20subgrIC} 
Murota, K., Tamura, A.:
Integrality of subgradients and biconjugates of integrally convex functions.
Optimization Letters
{\bf 14}, 195--208 (2020)



\bibitem{Sch86} 
Schrijver, A.:
Theory of Linear and Integer Programming.
 Wiley,  New York (1986)

\bibitem{Sch03}
Schrijver, A.:
Combinatorial Optimization---Polyhedra and Efficiency.
Springer, Heidelberg (2003)

\bibitem{Shi17L}
Shioura, A.:
Algorithms for L-convex function minimization:
Connection between discrete convex analysis and other research areas.
{Journal of the Operations Research Society of Japan} {\bf 60},
216--243 (2017)

%%No. 3



\bibitem{SCB14}
Simchi-Levi, D.,    Chen, X.,   Bramel, J.:
The Logic of Logistics: Theory, Algorithms, and Applications for Logistics Management, 3rd ed.
Springer, New York (2014)



\bibitem{TT21ddmc}
Tamura, A.,    Tsurumi, K.:
Directed discrete midpoint convexity.
Japan Journal of Industrial and Applied Mathematics {\bf 38}, 1--37 (2021)


\bibitem{Zie07} 
Ziegler, G.M.:
Lectures on Polytopes.
Springer, New York (1995);
Corrected and updated printing (2007)

\end{thebibliography}
\end{document}